\documentclass{article}
\usepackage[utf8]{inputenc}
\usepackage[T1]{fontenc}
\usepackage[english]{babel}
\usepackage{lmodern}
\usepackage{calrsfs}
\usepackage{color}
\usepackage{amsmath}
\usepackage{amsthm}
\usepackage{mathtools, bm}
\usepackage{amssymb, bm}
\usepackage{enumerate}
\usepackage{sectsty}
\usepackage{graphicx}
\usepackage{float}
\usepackage{cancel}
\usepackage{url}
\usepackage{nomencl}
\makenomenclature
\usepackage{makeidx}
\makeindex

\usepackage{here}
\usepackage{geometry}
\usepackage[pdftex]{hyperref}
\usepackage{pifont}
\usepackage{dsfont}

\usepackage{tikz}
\usetikzlibrary{matrix,arrows}
\usetikzlibrary{shapes}
\usetikzlibrary{calc}
\usetikzlibrary{arrows}
\usetikzlibrary{decorations.pathreplacing,decorations.markings}

\chapterfont{\centering}
\sectionfont{\centering}
\subsectionfont{}

\graphicspath{{GRAPHIQUES/}}

\newtheoremstyle{monstyledem} 
    {8pt}                    
    {8pt}                    
    {\normalfont}                   
    {}                           
    {\bf}                   
    {\newline}                          
    {.5em}                       
    {}  

\newtheoremstyle{monstyle} 
    {8pt}                    
    {8pt}                    
    {\itshape}                   
    {}                           
    {\bf}                   
    {\newline}                          
    {.5em}                       
    {}  

\theoremstyle{monstyle}

\newtheorem{thm}{Theorem}[section]
\newtheorem{de}{Definition}[section]
\newtheorem{prop}{Proposition}[section]
\newtheorem{lem}{Lemma}[section]
\newtheorem{cor}{Corollary}[section]
\newtheorem{rem}{Remark}[section]

\theoremstyle{monstyledem}

\newcommand\blankfootnote[1]{%
  \let\thefootnote\relax\footnotetext{#1}%
  \let\thefootnote\svthefootnote%
}

\title{Regular extensions and algebraic relations between values of Mahler functions in positive characteristic}
\author{Gwladys Fernandes }

\begin{document}
\maketitle

\blankfootnote{\begin{footnotesize}
This project has received funding from the European Research Council (ERC) under the European Union's Horizon 2020 research and innovation programme
under the Grant Agreement No 648132.
\end{footnotesize}}

\noindent
\textbf{{\footnotesize Abstract}} : {\footnotesize Let $\mathbb{K}$ be a function field of characteristic $p>0$. We recently established the analogue of a theorem of Ku. Nishioka for linear Mahler systems defined over $\mathbb{K}(z)$. This paper is dedicated to proving the following refinement of this theorem. Let $f_{1}(z),\ldots f_{n}(z)$ be $d$-Mahler functions such that $\overline{\mathbb{K}}(z)\left(f_{1}(z),\ldots, f_{n}(z)\right)$ is a regular extension over $\overline{\mathbb{K}}(z)$. Then, every homogeneous algebraic relation over $\overline{\mathbb{K}}$ between their values at a regular algebraic point arises as the specialization of a homogeneous algebraic relation over $\overline{\mathbb{K}}(z)$ between these functions themselves. If $\mathbb{K}$ is replaced by a number field, this result is due to B. Adamczewski and C. Faverjon, as a consequence of a theorem of P. Philippon. The main difference is that in characteristic zero, every $d$-Mahler extension is regular, whereas, in characteristic $p$, non-regular $d$-Mahler extensions do exist. Furthermore, we prove that the regularity of the field extension $\overline{\mathbb{K}}(z)\left(f_{1}(z),\ldots, f_{n}(z)\right)$ is also necessary for our refinement to hold. Besides, we show that, when $p\nmid d$, $d$-Mahler extensions over $\overline{\mathbb{K}}(z)$ are always regular. Finally, we describe some consequences of our main result concerning the transcendence of values of $d$-Mahler functions at algebraic points.}

\section{Introduction}

Let $\mathbb{K}$ be a field and let $d\geq 2$ be an integer. We say that a power series $f(z)\in \mathbb{K}[[z]]$ is a $d$-Mahler function over $\mathbb{K}(z)$ if there exist polynomials $P_{0}(z),\ldots, P_{n}(z)\in\mathbb{K}[z]$, $P_{n}(z)$ $\cancel\equiv$  $0$, such that
\begin{equation}
\label{equa_min}
P_{0}(z)f(z)+P_{1}(z)f(z^{d})+\cdots+P_{n}(z)f(z^{d^{n}})=0.
\end{equation}

The minimal integer $n$ satisfying the previous equation is called the order of $f(z)$.
We say that the column vector whose coordinates are the power series $f_{1}(z),\ldots, f_{n}(z)\in\mathbb{K}[[z]]$ satisfies a $d$-Mahler system if there exists a matrix $A(z) \in \text{GL}_{n}(\mathbb{K}(z))$ such that

\begin{equation}
\label{syst_gen}
\begin{pmatrix}
f_{1}(z^{d}) \\ \vdots \\ f_{n}(z^{d}) 
\end{pmatrix} =A(z)\begin{pmatrix}
f_{1}(z) \\ \vdots \\ f_{n}(z) 
\end{pmatrix}.
\end{equation}

Any $d$-Mahler function is a coordinate of a vector solution of the $d$-Mahler system associated with the companion matrix of \eqref{equa_min}. Reciprocally, every coordinate of a vector solution of a $d$-Mahler system is a $d$-Mahler function. We say that a number $\alpha\in \mathbb{K}$ is regular with respect to System \eqref{syst_gen} if for all integer $k\geq 0$, the number $\alpha^{d^{k}}$ is neither a pole of the matrix $A(z)$ nor a pole of the matrix $A^{-1}(z)$. In this paper, we are dealing with the case where $\mathbb{K}$ is a function field of positive characteristic. Let us introduce the associated framework. We start with a prime number $p$ and a power of $p$ denoted by $q=p^{r}$. Then, we let $A=\mathbb{F}_{q}[T]$ denote the ring of polynomials in $T$, with coefficients in the finite field $\mathbb{F}_{q}$, and we let $K=\mathbb{F}_{q}(T)$ denote the fraction field of $A$. We define the $\frac{1}{T}$-adic absolute value on $K$ by $\left|\frac{P(T)}{Q(T)}\right|=q^{\text{deg}_{T}(P)-\text{deg}_{T}(Q)}$. We recall that the completion of $K$ with respect to $|.|$ is the field $\mathbb{F}_{q}\left(\left(\frac{1}{T}\right)\right)$ of Laurent power series expansions over $\mathbb{F}_{q}$, and that the completion $C$ of the algebraic closure of $\mathbb{F}_{q}\left(\left(\frac{1}{T}\right)\right)$ with respect to the unique extension of $|.|$ is a complete and algebraically closed field. Finally, as announced, we let $\mathbb{K}$ denote a function field, that is, a finite extension of $K$. We let $\overline{\mathbb{K}}$ denote the algebraic closure of $\mathbb{K}$, embedded in $C$.

Let $\mathbb{K}\{z\}$ denote the set of functions which admit a convergent power series expansion in a domain containing the origin, with coefficients in $\mathbb{K}$. Let $\bm{k}$ be a field and $\mathcal{F}$ a family of elements of a $\bm{k}$-algebra. We let $\text{trdeg}_{\bm{k}}\{\mathcal{F}\}$ denote the transcendence degree of $\mathcal{F}$ over $\bm{k}$. That is, the maximal number of elements of $\mathcal{F}$ that are algebraically independent over $\bm{k}$. In \cite{Fernandes}, the author proves the following result. This is the analogue for function fields of characteristic $p$ of a classical result due to Ku. Nishioka \cite{N-art} when $\mathbb{K}$ is a number field.

\begin{thm}[F.]
\label{nish_gen}
Let $n\geq 1$, $d\geq 2$ be two integers and $f_{1}(z), \ldots, f_{n}(z) \in \mathbb{K}\{z\}$ be functions satisfying $d$-Mahler System \eqref{syst_gen}. Let $\alpha\in\overline{\mathbb{K}}$, $0<|\alpha|<1$, be a regular number with respect to System \eqref{syst_gen}.
Then
\begin{equation}
\label{egdegtr_gen}
\text{trdeg}_{\overline{\mathbb{K}}}\{f_{1}(\alpha), \ldots, f_{n}(\alpha)\}=\text{trdeg}_{\overline{\mathbb{K}}(z)}\{f_{1}(z), \ldots, f_{n}(z)\}.
\end{equation}
 \end{thm}

In general, few is known about the algebraic relations between the functions $f_{1}(z), \ldots, f_{n}(z)$ over $\overline{\mathbb{K}}(z)$. This makes a priori difficult the question to decide whether $f(\alpha)$ is transcendental or not over $\overline{\mathbb{K}}$. However, it is easier to study linear relations between the functions $f_{1}(z), \ldots, f_{n}(z)$ over $\overline{\mathbb{K}}(z)$. For example, when $\mathbb{K}$ is a number field, a basis of the set of linear relations over $\overline{\mathbb{Q}}(z)$ between the Mahler functions $f_{1}(z), \ldots, f_{n}(z)$ can be explicitly computed \cite{A-F,A-F_effectif}. The arguments used by B. Adamczewski and C. Faverjon to obtain this result belong to linear algebra and might fit for function fields. This could be a further perspective of study. For these reasons, we are interested in refining Theorem \ref{nish_gen}. Let $\bm{k}$ be a field. We say that a finitely generated field extension $\mathcal{E}=\bm{k}(u_{1}, \ldots, u_{n})$ of $\bm{k}$ is regular over $\bm{k}$ if the two following conditions are satisfied.
\begin{enumerate}
	\item
	$\mathcal{E}$ is separable over $\bm{k}$. That is, there exists a transcendence basis $\mathcal{F}$ of $\mathcal{E}$ over $\bm{k}$ such that $\mathcal{E}$ is a separable algebraic extension of $\bm{k}(\mathcal{F})$ (see \cite[Appendix A1.2]{Eisenbud} and also \cite{Maclane}).
	\item
	Every element of $\mathcal{E}$ that is algebraic over $\bm{k}$ belongs to $\bm{k}$.
\end{enumerate}

With this definition, our main result is the following.

\begin{thm}
\label{phil_p}
We continue with the assumptions of Theorem \ref{nish_gen}. Let us assume further that the extension \\$\overline{\mathbb{K}}(z)(f_{1}(z), \ldots, f_{n}(z))$ is regular over $\overline{\mathbb{K}}(z)$. 
 
Then, for all polynomial $P(X_{1},\ldots, X_{n})\in\overline{\mathbb{K}}[X_{1},\ldots,X_{n}]$ homogeneous in $X_{1},\ldots,X_{n}$ such that 
$$P(f_{1}(\alpha),\ldots, f_{n}(\alpha))=0,$$
there exists a polynomial $Q(z,X_{1},\ldots, X_{n})\in \overline{\mathbb{K}}[z][X_{1},\ldots,X_{n}]$ homogeneous in $X_{1},\ldots,X_{n}$ such that 
$$Q(z,f_{1}(z),\ldots, f_{n}(z))=0,$$
and 
$$Q(\alpha,X_{1},\ldots, X_{n})=P(X_{1},\ldots, X_{n}).$$
\end{thm}

Let us note that any inhomogeneous algebraic relation
$$P(f_{1}(\alpha),\ldots, f_{n}(\alpha))=0$$
can be turned into a homogeneous algebraic relation between the values at $\alpha$ of the functions $f_{i}(z)$ and the additional function 1.
\medskip

As announced, Theorem \ref{phil_p} allows us to deal with linear independence over $\overline{\mathbb{K}}$ between values of Mahler functions.

\begin{cor}
	\label{cor_direct}
We continue with the assumptions of Theorem \ref{phil_p}. If the functions $f_{1}(z), \ldots, f_{n}(z)$ are linearly independent over $\overline{\mathbb{K}}(z)$, then, the numbers $f_{1}(\alpha), \ldots, f_{n}(\alpha)$ are linearly independent over $\overline{\mathbb{K}}$.
\end{cor}

Given $f(z)$ a Mahler function, one of the main goals of Mahler's method is to decide whether $f(\alpha)$ is transcendental or not over $\overline{\mathbb{K}}$. Corollary \ref{cor_direct} applied with the functions $1,f(z)$ shows the contribution of Theorem \ref{phil_p} in understanding the nature of $f(\alpha)$ when $\alpha$ is regular. Corollary \ref{cor_important} below states that this contribution even extends to the case of non-regular numbers $\alpha$. Let us start with a single transcendental $d$-Mahler function $f(z)$. Then, there exist an integer $m\geq 1$ and coprime polynomials $P_{-1}(z),\ldots, P_{m}(z)\in\mathbb{K}[z]$, $P_{m}(z)$ $\cancel\equiv$  $0$, such that
\begin{equation}
\label{eq_inh_min}
P_{-1}(z)+P_{0}(z)f(z)+P_{1}(z)f(z^{d})+\cdots+P_{m}(z)f(z^{d^{m}})=0. 
\end{equation}

If $m$ is minimal, we call \eqref{eq_inh_min} the minimal inhomogeneous equation of $f(z)$ over $\mathbb{K}(z)$. We can associate with this equation the $d$-Mahler system
\begin{equation}
\label{syst_inh}
\begin{pmatrix}
1 \\ f(z^{d}) \\ \vdots \\ f(z^{d^{m}}) 
\end{pmatrix} =A(z)\begin{pmatrix}
1 \\ f(z) \\ \vdots \\ f(z^{d^{m-1}}) 
\end{pmatrix},
\end{equation}
where $A(z)\in \text{GL}_{m+1}(\mathbb{K}(z))$ is the companion matrix of Equation \eqref{eq_inh_min}. Then, let us write $\sigma_{d}$ to denote the endomorphism of $\mathbb{K}\{z\}$ defined by $\sigma_{d}g(z)=g(z^{d})$. Then, we set 
$$\overline{\mathbb{K}}(z)(g(z))_{\sigma_{d}}=\overline{\mathbb{K}}(z)\left(\{\sigma_{d}^{i}g(z)\}_{i\geq 0}\right).$$ 
Now, let $\alpha\in\overline{\mathbb{K}}$, $0<|\alpha|<1$, be a regular number for System \eqref{syst_inh}. The only thing we know a priori is that
$$\text{trdeg}_{\overline{\mathbb{K}}(z)}\{1, f(z),\ldots,f(z^{d^{m-1}}) \}\geq 1.$$
Therefore, Theorem \ref{nish_gen} only gives
$$\text{trdeg}_{\overline{\mathbb{K}}}\{1, f(\alpha),\ldots,f(\alpha^{d^{m-1}})\}\geq 1.$$
That is, there exists at least one transcendental number among $f(\alpha),\ldots,f(\alpha^{d^{m-1}})$. But we cannot conclude that $f(\alpha)$ is transcendental. Our contribution to this problem is the following result.

\begin{cor}
	\label{cor_important}
	Let $f(z)\in\mathbb{K}\{z\}$ be a $d$-Mahler transcendental function over $\mathbb{K}(z)$. Let $\alpha\in\overline{\mathbb{K}}$, $0<|\alpha|<1$ such that $\alpha$ is in the disc of convergence of $f(z)$. Let us assume that the extension $\overline{\mathbb{K}}(z)(f(z))_{\sigma_{d}}$ is regular over $\overline{\mathbb{K}}(z)$.

	Then, we have the following.
	\begin{enumerate}
		\item 
		The number $f(\alpha)$ is either transcendental or in $\mathbb{K}(\alpha)$.
		\item
		If $\alpha$ is a regular number with respect to $d$-Mahler System \eqref{syst_inh} satisfied by $f(z)$ (that is $P_{0}(\alpha^{d^{k}})P_{m}(\alpha^{d^{k}})\neq 0$ for all integer $k\geq 0$), then $f(\alpha)$ is transcendental over $\overline{\mathbb{K}}$.
	\end{enumerate}
\end{cor}

Such results were first established in the setting of linear differential equations over $\overline{\mathbb{Q}}(z)$, especially for E-functions. Theorem \ref{nish_gen} is the analogue of Siegel-Shidlovskii's Theorem \cite{S-S}. Theorem \ref{phil_p} is the analogue of a theorem of F. Beukers \cite{Beukers}. F. Beukers's proof uses Galois Theory and results from Y. André. Moreover, Y. André proved \cite{andre} that the theorem of F. Beukers can be deduced from Siegel-Shidlovskii's theorem, using a new method involving the theory of affine quasi-homogeneous varieties. Finally, the analogue of Corollary \ref{cor_important} for E-functions is stated in \cite{Rivoal-Fischler} (see also \cite{Boris-Tanguy}). Getting back to Mahler functions, Theorem \ref{phil_p} is the analogue for function fields of a theorem of B. Adamczewski and C. Faverjon \cite{A-F}, obtained as a consequence of a result of P. Philippon \cite{Phil1}. The analogues of Corollary \ref{cor_direct} and Corollary \ref{cor_important} for number fields are proved in \cite{A-F}.

Besides, if $f_{1}(z),\ldots, f_{n}(z)$ are either $E$-functions or Mahler functions over $\overline{\mathbb{Q}}(z)$, the extension\\ $\overline{\mathbb{Q}}(z)\left(f_{1}(z),\ldots, f_{n}(z)\right)$ is always regular over $\overline{\mathbb{Q}}(z)$. This is straightforward for E-functions for they are analytic in the whole complex plane. For Mahler functions, this can be deduced  \cite{A-F,Phil1} from the fact that a Mahler function with coefficients in $\overline{\mathbb{Q}}$ is either rational or transcendental \cite[Theorem 5.1.7]{N}).  But when $\mathbb{K}$ is a function field of characteristic $p$, such a dichotomy does not hold anymore and there do exist non-regular Mahler extensions. Let us provide a trivial example based on the following $p$-Mahler system.
  \begin{equation*}  
  \begin{pmatrix}
  f_{1}(z^{p}) \\ f_{2}(z^{p})
  \end{pmatrix}=\begin{pmatrix}
  1 & 0 \\ -z & 1
  \end{pmatrix}\begin{pmatrix}
  f_{1}(z) \\ f_{2}(z)
  \end{pmatrix}.  
  \end{equation*}

 A solution to this system is given by
 $$f_{1}(z)=1, f_{2}(z)=\sum_{n=0}^{+\infty}z^{p^{n}}.$$
 Furthermore, $f_{2}(z)$ is algebraic because $f_{2}(z)^{p}=f_{2}(z^{p})=f_{2}(z)-z$. On the other hand, the sequence of coefficients of $f_{2}(z)$ is not eventually periodic. Therefore, $f_{2}(z)$ is not rational. It follows that the extension $\mathcal{E}=\overline{\mathbb{K}}(z)(f_{1}(z),f_{2}(z))$ is not regular over $\overline{\mathbb{K}}(z)$. Now, let $\alpha\in\overline{\mathbb{K}}$, $0<|\alpha|<1$ and $\lambda=f_{2}(\alpha)\in\overline{\mathbb{K}}$. Then, $\lambda f_{1}(\alpha)-f_{2}(\alpha)=0$ is a non-trivial linear relation between $f_{1}(\alpha)$ and $f_{2}(\alpha)$ over $\overline{\mathbb{K}}$. However, there is no non-trivial linear relation between the function $f_{1}(z)$ and $f_{2}(z)$ over $\overline{\mathbb{K}}(z)$, because $f_{2}(z)$ is not rational. Hence, the conclusion of Theorem \ref{phil_p} does not hold in this case. In Theorem \ref{thm_faux}, we state that this example reflects a general behaviour. That is, the conclusion of Theorem \ref{phil_p} is never satisfied when the extension $\overline{\mathbb{K}}(z)\left(f_{1}(z),\ldots, f_{n}(z)\right)$ is not regular over $\overline{\mathbb{K}}(z)$. Let us first introduce some definitions and notations. 
 Let $\bm{k}$ be a valued field and $\bm{k}_{c}$ its completion. Note that its valuation extends uniquely to $\overline{\bm{k}_{c}}$ \cite[II.2, Corollary 2]{Serre}. We let $\tilde{\bm{k}}$ denote the completion of $\overline{\bm{k}_{c}}$ with respect to this valuation. Then, $\tilde{\bm{k}}$ is complete and algebraically closed. Now, let $\alpha\in\tilde{\bm{k}}$. We say that a function is analytic at $\alpha$ if it admits a convergent power series expansion in a connected open neighbourhood of $\alpha$, with coefficients in $\tilde{\bm{k}}$. If $\mathcal{U}\subseteq\tilde{\bm{k}}$ is a domain, we say that a function is analytic on $\mathcal{U}$ if it is analytic at each point of $\mathcal{U}$. If the power series expansion of $f(z)$ at $\alpha\in\mathcal{U}$ has coefficients in a sub-field $L$ of $\tilde{\bm{k}}$, we say that $f(z)$ is analytic at $\alpha$ over $L$ and denote the set of all such functions by $L\{z-\alpha\}$. Now, let $f_{1}(z),\ldots, f_{n}(z)\in\bm{k}\{z\}$. We set 
$$\mathfrak{p}=\{Q(z,X_{1},\ldots,X_{n})\in\overline{\bm{k}}(z)[X_{1},\ldots,X_{n}], Q(z,f_{1}(z),\ldots,f_{n}(z))=0\}.$$
If the functions $f_{1}(z),\ldots, f_{n}(z)$ are analytic at $\alpha\in\overline{\bm{k}}$, we set 
$$\mathfrak{p}_{\alpha}=\{P(X_{1},\ldots,X_{n})\in\overline{\bm{k}}[X_{1},\ldots,X_{n}], P(f_{1}(\alpha),\ldots,f_{n}(\alpha))=0\}.$$ 
Let $R$ be a ring. If $\mathfrak{q}$ is an ideal of $A=R[X_{1},\ldots, X_{n}]$, we write $\tilde{\mathfrak{q}}$ to refer to the homogenized ideal of $\mathfrak{q}$. It is the ideal of $A'=R[X_{0},X_{1},\ldots, X_{n}]$ generated by all the homogeneous polynomials $Q(X_{0},\ldots, X_{n})\in A'$ for which there exists a polynomial $P(X_{1},\ldots, X_{n})\in\mathfrak{q}$ such that $Q(1,X_{1},\ldots, X_{n})=P(X_{1},\ldots, X_{n})$. Finally, let $\text{ev}_{\alpha}(\tilde{\mathfrak{p}}\cap \overline{\bm{k}}[z][X_{0},\ldots, X_{n}])$ denote the homogeneous ideal over $\overline{\bm{k}}[X_{0},\ldots, X_{n}]$ constructed by evaluating the ideal $\tilde{\mathfrak{p}}\cap \overline{\bm{k}}[z][X_{0},\ldots, X_{n}]$ at $z=\alpha$.
With these definitions, the conclusion of Theorem \ref{phil_p} is equivalent to the following assertion.
$$\text{ev}_{\alpha}(\tilde{\mathfrak{p}}\cap \overline{\bm{k}}[z][X_{0},\ldots, X_{n}])=\tilde{\mathfrak{p}}_{\alpha}.$$
Now, we can state the announced result.

\begin{thm}
	\label{thm_faux}
	Let $\bm{k}$ be a valued field and let us assume that $f_{1}(z),\ldots, f_{n}(z)\in\overline{\bm{k}}\{z\}$ are analytic functions in a domain $\mathcal{U}\subseteq\tilde{\bm{k}}$ which contains the origin. Let $\alpha\in\mathcal{U}\cap\overline{\bm{k}}$. Let us assume further that the extension $\overline{\bm{k}}(z)(f_{1}(z), \ldots, f_{n}(z))$ is not regular over $\overline{\bm{k}}(z)$. Then we have
	$$\text{ev}_{\alpha}(\tilde{\mathfrak{p}}\cap \overline{\bm{k}}[z][X_{0},\ldots, X_{n}])\varsubsetneq\tilde{\mathfrak{p}}_{\alpha}.$$
	In other words, the conclusion of Theorem \ref{phil_p} does not hold.
\end{thm}

Besides, let $f_{1}(z),\ldots, f_{n}(z)\in\mathbb{K}\{z\}$ be $d$-Mahler functions over $\mathbb{K}(z)$. Then, we say that the field extension $\overline{\mathbb{K}}(z)\left(f_{1}(z),\ldots, f_{n}(z)\right)$ is a $d$-Mahler extension over $\overline{\mathbb{K}}(z)$, or, for short, $d$-Mahler. Now, if $p\nmid d$, we show that $d$-Mahler extensions over $\overline{\mathbb{K}}(z)$ behave just as in characteristic zero.
 
 \begin{thm}
 	\label{l_mahler_bis}
 	Let $d\geq 2$ be an integer such that $p\nmid d$. Then, a $d$-Mahler function $f(z)\in\mathbb{K}\{z\}$ over $\mathbb{K}(z)$ is either transcendental or in $\mathbb{K}(z)$.
 \end{thm}

 \begin{cor}
 	\label{l_mahler}
 	Let $d\geq 2$ be an integer such that $p\nmid d$. Then, a $d$-Mahler extension over $\overline{\mathbb{K}}(z)$ is always regular over $\overline{\mathbb{K}}(z)$.
 \end{cor}

The present paper is organised as follows. Section \ref{sec_phil} is dedicated to the proof of Theorem \ref{phil_p}. We follow the same approach as P. Philippon \cite{Phil1} and B. Adamczewski and C. Faverjon \cite{A-F}. In Section \ref{faux} we prove Theorem \ref{thm_faux}. Section \ref{sec_courbe} is devoted to the proof of Theorem \ref{l_mahler_bis} and Corollary \ref{l_mahler}. We follow an approach of J. Roques \cite{JR} dealing with the theory of smooth projective curves in $\mathbb{P}^{1}(C)$, and an argument from B. Adamczewski and C. Faverjon \cite{A-F}. In Section \ref{cor}, we prove Corollary \ref{cor_important}. Finally, in Section \ref{Examples} we give an application of Theorem \ref{phil_p} and provide, in the case where $p\mid d$, examples of regular and non-regular $d$-Mahler extensions.

\section{Proof of Theorem \ref{phil_p}}
\label{sec_phil}

Before going through the proof of Theorem \ref{phil_p}, let us introduce some definitions and recall some results. Let $L$ be a field. If $\mathfrak{q}$ is a prime ideal of $L[X_{1},\ldots, X_{n}]$, we say that $\mathfrak{q}$ is absolutely prime over $L$ if for all extension $L_{1}$ of $L$, the extended ideal $\mathfrak{q}L_{1}[X_{1},\ldots, X_{n}]$ is still prime in $L_{1}[X_{1},\ldots, X_{n}]$. We recall that $\mathfrak{q}$ is prime (resp. absolutely prime) in $L[X_{1},\ldots, X_{n}]$ if and only if its homogenized ideal $\tilde{\mathfrak{q}}$ is prime (resp. absolutely prime) in $L[X_{0},\ldots, X_{n}]$. Furthermore, when both are prime, they have the same height. Finally, given functions $f_{1}(z), \ldots, f_{n}(z) \in \mathbb{K}\{z\}$, we recall that the extension $\overline{\mathbb{K}}(z)\left(f_{1}(z),\ldots, f_{n}(z)\right)$ is regular over $\overline{\mathbb{K}}(z)$ if and only if the ideal $\mathfrak{p}$ is absolutely prime in $\overline{\mathbb{K}}(z)\left[X_{1},\ldots, X_{n}\right]$ \cite[VII, Theorem 39]{Sam_Zar_2}.

\subsection{A local version of Theorem \ref{phil_p}}
\label{step1}
In this subsection, we establish an analogue of a result of P. Philippon \cite[Prop. 4.4]{Phil1} in the framework of function fields. This is Corollary \ref{phil_local} below. We deduce this statement from the more general result stated in Proposition \ref{phil_local_ens}, in the vein of \cite[Proposition 3.1]{A-F}.


\begin{prop}
\label{phil_local_ens}
Let $\bm{k}$ be a valued field and let $f_{1}(z),\ldots, f_{n}(z)\in\bm{k}\{z\}$ be analytic functions on a domain $\mathcal{U}\subseteq\tilde{\bm{k}}$ which contains the origin, over $\tilde{\bm{k}}$. Let us assume that the two following properties are satisfied.
\begin{enumerate}
\item
There exists a set $S\subseteq\mathcal{U}\cap\overline{\bm{k}}$ such that, for all $\alpha\in S$, we have
\begin{equation}
\label{nish_p_local}
\text{trdeg}_{\overline{\bm{k}}}\{f_{1}(\alpha), \ldots, f_{n}(\alpha)\}=\text{trdeg}_{\overline{\bm{k}}(z)}\{f_{1}(z), \ldots, f_{n}(z)\}.
\end{equation} 
\item
The extension $\overline{\bm{k}}(z)(f_{1}(z), \ldots, f_{n}(z))$ is regular over $\overline{\bm{k}}(z)$. 
\end{enumerate} 

Then, there exists a finite set $S'\subseteq S$ such that for all $\alpha\in S\setminus S'$ and for all polynomial $P(X_{1},\ldots, X_{n})\in\overline{\bm{k}}[X_{1},\ldots,X_{n}]$ of total degree $N$ in $X_{1},\ldots,X_{n}$ such that 
$$P(f_{1}(\alpha),\ldots, f_{n}(\alpha))=0,$$
there exists a polynomial $Q(z,X_{1},\ldots, X_{n})\in \overline{\bm{k}}[z][X_{1},\ldots,X_{n}]$ of total degree $N$ in $X_{1},\ldots,X_{n}$ such that 
$$Q(z,f_{1}(z),\ldots, f_{n}(z))=0,$$
and 
$$Q(\alpha,X_{1},\ldots, X_{n})=P(X_{1},\ldots, X_{n}).$$
\end{prop}

\begin{cor}
	\label{phil_local}
Let $f_{1}(z), \ldots, f_{n}(z)\in\mathbb{K}\{z\}$ be functions satisfying \eqref{syst_gen} and such that the extension\\  $\overline{\mathbb{K}}(z)(f_{1}(z), \ldots, f_{n}(z))$ is regular over $\overline{\mathbb{K}}(z)$. Then, there exists $0<\rho<1$ such that for all $\alpha\in\overline{\mathbb{K}}$, $0<|\alpha|<\rho$, and for all polynomial $P(X_{1},\ldots, X_{n})\in\overline{\mathbb{K}}[X_{1},\ldots,X_{n}]$ of total degree $N$ in $X_{1},\ldots,X_{n}$ such that 
$$P(f_{1}(\alpha),\ldots, f_{n}(\alpha))=0,$$
there exists a polynomial $Q(z,X_{1},\ldots, X_{n})\in \overline{\mathbb{K}}[z][X_{1},\ldots,X_{n}]$ of total degree $N$ in $X_{1},\ldots,X_{n}$ such that 
$$Q(z,f_{1}(z),\ldots, f_{n}(z))=0,$$
and 
$$Q(\alpha,X_{1},\ldots, X_{n})=P(X_{1},\ldots, X_{n}).$$ 	
	
\end{cor}

\begin{proof}[Proof of Proposition \ref{phil_local_ens}]

As noticed in the introduction of this paper, the conclusion of Proposition \ref{phil_local_ens} is equivalent to the following
$$\text{ev}_{\alpha}(\tilde{\mathfrak{p}}\cap \overline{\bm{k}}[z][X_{0},\ldots, X_{n}])=\tilde{\mathfrak{p}}_{\alpha},$$
for all but finitely many $\alpha\in S$.

 Thus, proving Proposition \ref{phil_local_ens} is the same as proving that $\text{ev}_{\alpha}(\tilde{\mathfrak{p}}\cap\overline{\bm{k}}[z][X_{0},\ldots, X_{n}])$ is a prime ideal of same height as $\tilde{\mathfrak{p}}_{\alpha}$, for all but finitely many $\alpha\in S$. To do so, we notice that the ring $\overline{\bm{k}}(z)[f_{1}(z), \ldots, f_{n}(z)]$ is an integral (because $\mathcal{U}$ is a domain) finitely generated $\overline{\bm{k}}(z)$-algebra. Hence applying results from commutative algebra (which only rely on these two properties and hold true over any base field, see for example \cite{Eisenbud}), we get

\begin{align*}
\text{trdeg}_{\overline{\bm{k}}(z)}\{f_{1}(z), \ldots, f_{n}(z)\} &=\text{dim}\left(\overline{\bm{k}}(z)[f_{1}(z), \ldots, f_{n}(z)]\right)\\
&=\text{dim}\left(\overline{\bm{k}}(z)[X_{1}, \ldots, X_{n}]/\mathfrak{p}\right)\\
&=\text{dim}\left(\overline{\bm{k}}(z)[X_{1}, \ldots, X_{n}]\right)-\text{ht}(\mathfrak{p})\\
&=\text{dim}\left(\overline{\bm{k}}(z)[X_{0}, \ldots, X_{n}]\right)-\text{ht}(\tilde{\mathfrak{p}})-1.
\end{align*}

Let $\alpha\in S$. As $\overline{\bm{k}}[f_{1}(\alpha), \ldots, f_{n}(\alpha)]$ is an integral finitely generated $\overline{\bm{k}}$-algebra, we obtain in the same way that

$$\text{trdeg}_{\overline{\bm{k}}}\{f_{1}(\alpha), \ldots, f_{n}(\alpha)\}=\text{dim}\left(\overline{\bm{k}}[X_{0}, \ldots, X_{n}]\right)-\text{ht}(\tilde{\mathfrak{p}}_{\alpha})-1.$$

By assumption, we get
\begin{equation}
\label{eg_ht}
\text{ht}(\tilde{\mathfrak{p}})=\text{ht}(\tilde{\mathfrak{p}}_{\alpha}).
\end{equation}

Thus, proving Proposition \ref{phil_local_ens} is now equivalent to prove that $\text{ev}_{\alpha}(\tilde{\mathfrak{p}}\cap \overline{\bm{k}}[z][X_{0},\ldots, X_{n}])$ is a prime ideal of same height as $\tilde{\mathfrak{p}}$, for all but finitely many $\alpha\in S$. First, as, by assumption, the extension $\overline{\bm{k}}(z)(f_{1}(z), \ldots, f_{n}(z))$ is regular over $\overline{\bm{k}}(z)$, the ideal $\mathfrak{p}$ is absolutely prime over $\overline{\bm{k}}(z)[X_{1}, \ldots, X_{n}]$ \cite[VII, Theorem 39]{Sam_Zar_2}. Therefore, as recalled earlier, $\tilde{\mathfrak{p}}$ is absolutely prime over $\overline{\bm{k}}(z)[X_{0}, \ldots, X_{n}]$. Now, a result from W. Krull \cite{KrullII}, which holds for any base field, leads to the existence of a finite set $S'\subseteq S$ such that for all $\alpha\in S\setminus S'$, the ideal $\text{ev}_{\alpha
}\left(\tilde{\mathfrak{p}}\cap \overline{\bm{k}}[z][X_{0}, \ldots, X_{n}]\right)$ is absolutely prime over $\overline{\bm{k}}[X_{0}, \ldots, X_{n}]$. In particular, it is a prime ideal. Finally, let $\alpha\in S\setminus S'$. To prove that

\begin{equation}
\label{fin_en_sachant_premier}
\text{ht}\left(\text{ev}_{\alpha}\left(\tilde{\mathfrak{p}}\cap \overline{\bm{k}}[z][X_{0},\ldots, X_{n}]\right)\right)=\text{ht}\left(\tilde{\mathfrak{p}}\right),
\end{equation}

\noindent 
we first notice that 
	
\begin{equation}
\text{ev}_{\alpha}(\tilde{\mathfrak{p}}\cap \overline{\bm{k}}[z][X_{0},\ldots, X_{n}])\subseteq\tilde{\mathfrak{p}}_{\alpha}.
\end{equation}

It follows that

\begin{align*}
\text{ht}\left(\text{ev}_{\alpha}\left(\tilde{\mathfrak{p}}\cap \overline{\bm{k}}[z][X_{0},\ldots, X_{n}]\right)\right)&\leq\text{ht}\left(\tilde{\mathfrak{p}}_{\alpha}\right)\\
&=\text{ht}\left(\tilde{\mathfrak{p}}\right), \text{ by }  \eqref{eg_ht}.
\end{align*}

In order to prove the converse inequality, we use a result of D. Hilbert (see for example \cite[VII, Theorem 41, Theorem 42]{Sam_Zar_2}). We give a detail account here because we did not find a reference in print. We reproduce an argument due to C. Faverjon (unpublished). We first introduce the following definitions, according to \cite{Nest_shid}.

\begin{de}
\begin{enumerate}
\item
For every $N\in\mathbb{N}$ and every homogeneous ideal $I$ of $\overline{\bm{k}}[X_{0},\ldots, X_{n}]$, let us set
$$\mathcal{M}_{I}(N)=\text{vect}_{\overline{\bm{k}}}\{[P]_{I}, P\in\overline{\bm{k}}[X_{0},\ldots, X_{n}], \text{ homogeneous of degree } N\},$$
where $[P]_{I}$ stands for the congruence class of $P$ modulo $I$.
\item
For every $N\in\mathbb{N}$ and every homogeneous ideal $J$ of $\overline{\bm{k}}(z)[X_{0},\ldots, X_{n}]$, let us set :
$$\mathcal{L}_{J}(N)=\text{vect}_{\overline{\bm{k}}(z)}\{[Q]_{J}, Q\in\overline{\bm{k}}(z)[X_{0},\ldots, X_{n}], \text{ homogeneous of degree } N\},$$
where $[Q]_{J}$ stands for the congruence class of $Q$ modulo $J$.
\end{enumerate}
\end{de}

Now, let us set
$$dim_{\overline{\bm{k}}}\left(\mathcal{M}_{\text{ev}_{\alpha}\left(\tilde{\mathfrak{p}}\cap \overline{\bm{k}}[z][X_{0},\ldots, X_{n}]\right)}(N)\right)=\phi(N),$$
and :
$$dim_{\overline{\bm{k}}(z)}\left(\mathcal{L}_{\tilde{\mathfrak{p}}}(N)\right)=\psi(N).$$

Then, we recall the following result.

\begin{thm}[D. Hilbert] 
\label{thm_hilbert}
For all integers $N\geq 0$, the quantities $\phi(N)$ and $\psi(N)$ are finite. Moreover, for all $N$ big enough, they are polynomials in $N$, and there exist $a,b>0$ such that
\begin{equation}
\label{asymt_ht}
\begin{cases}
\phi(N)&\sim_{N\rightarrow +\infty} aN^{n-\text{ht}\left(\text{ev}_{\alpha}\left(\tilde{\mathfrak{p}}\cap \overline{\bm{k}}[z][X_{0},\ldots, X_{n}]\right)\right)}\\
\psi(N)&\sim_{N\rightarrow +\infty} bN^{n-\text{ht}\left(\tilde{\mathfrak{p}}\right)}
\end{cases}
\end{equation}
\end{thm}

With this theorem in hands, we only need to prove that
\begin{equation}
\label{a_dem}
\phi(N)\leq \psi(N),
\end{equation}
for $N$ large enough. We now set the following definition. For all $\alpha\in\overline{\bm{k}}$, we denote by $R_{\alpha}$ the localization of the ring $\overline{\bm{k}}[z]$ at the ideal $(z-\alpha)$. In other words, the sub-field of $\overline{\bm{k}}(z)$ consisting of rational fractions without pole at $z=\alpha$. Then, the result \cite[Lemma 3]{Nest_shid} furnishes polynomials $b_{1}(z),\ldots, b_{\psi(N)}(z)\in \overline{\bm{k}}[z][X_{0},\ldots, X_{n}]$ such that $\mathcal{B}=\{[b_{1}(z)]_{\tilde{\mathfrak{p}}},\ldots, [b_{\psi(N)}(z)]_{\tilde{\mathfrak{p}}}\}$ is an $\alpha$-basis of $\mathcal{L}_{\tilde{\mathfrak{p}}}(N)$ over $\overline{\bm{k}}(z)$. That is, the following two properties are satisfied.
\begin{enumerate}
\item[(i)]
$\mathcal{B}$ is a $\overline{\bm{k}}(z)$-basis of $\mathcal{L}_{\tilde{\mathfrak{p}}}(N)$
\item[(ii)]
Every residue modulo $\tilde{\mathfrak{p}}$ of a homogeneous polynomial of degree $N$ in $R_{\alpha}[X_{0},\ldots, X_{n}]$ is a linear combination of $[b_{1}(z)]_{\tilde{\mathfrak{p}}},\ldots, [b_{\psi(N)}(z)]_{\tilde{\mathfrak{p}}}$, with coefficients in $R_{\alpha}$.
\end{enumerate}

This result is used by Y. V. Nesterenko and A. B. Shidlovskii for the field $\mathbb{C}$ instead of $\overline{\bm{k}}$. In our case, we can, as these authors, notice that the finite set of residues modulo $\tilde{\mathfrak{p}}$ of all monomials $$X_{0}^{i_{0}}\ldots X_{n}^{i_{n}}$$ of degree $i_{0}+\cdots+i_{n}= N$ generates $\mathcal{L}_{\tilde{\mathfrak{p}}}(N)$ over $\overline{\bm{k}}(z)$ and satisfies Property (ii). Among all such finite sets which generate $\mathcal{L}_{\tilde{\mathfrak{p}}}(N)$ and satisfy Property (ii), let us consider a set $\mathcal{S}=\{S_{1}(z),\ldots, S_{s}(z)\}$ whose cardinality is minimal. If $\mathcal{S}$ does not satisfy Property (i), there exist coprime polynomials $T_{1}(z),\ldots, T_{s}(z)\in\overline{\bm{k}}[z]$ such that
$$\sum_{l=1}^{s}T_{l}(z)S_{l}(z)=0.$$
Without any loss of generality, we may assume that $T_{s}(\alpha)\neq 0$. It follows that
$$S_{s}(z)=-\sum_{l=1}^{s-1}\frac{T_{l}(z)}{T_{s}(z)}S_{l}(z).$$
This contradicts the minimality of $s$. Thus, \cite[Lemma 3]{Nest_shid} remains true in our framework. 

\begin{rem}
	\label{alpha_base_mod}
	\em{
We see that the proof guarantees that we can choose an $\alpha$-basis of $\mathcal{L}_{\tilde{\mathfrak{p}}}(N)$ over $\overline{\bm{k}}(z)$ among the set of residues modulo $\tilde{\mathfrak{p}}$ of all monic monomials $$X_{0}^{i_{0}}\ldots X_{n}^{i_{n}}$$ of degree $i_{0}+\cdots+i_{n}= N$.}	
\end{rem}

\medskip

Now, we are going to show that the family

$$\left\{[\text{ev}_{\alpha}(b_{i}(z))]_{\text{ev}_{\alpha}\left(\tilde{\mathfrak{p}}\cap \overline{\bm{k}}[z][X_{0},\ldots, X_{n}]\right)}\right\}_{1\leq i\leq\psi(N)}$$
generates $\mathcal{M}_{\text{ev}_{\alpha}\left(\tilde{\mathfrak{p}}\cap \overline{\bm{k}}[z][X_{0},\ldots, X_{n}]\right)}(N)$ over $\overline{\bm{k}}$. Let $P(X_{0},\ldots,X_{n})\in\overline{\bm{k}}[X_{0},\ldots, X_{n}]$ be a homogeneous polynomial of degree $N$. As $\overline{\bm{k}}\subseteq R_{\alpha}$, there exist elements $r_{1},\ldots,r_{\psi(N)}\in R_{\alpha}$ such that
\begin{equation}
\label{base}
P(X_{0},\ldots,X_{n})-\sum_{i=1}^{\psi(N)}r_{i}b_{i}(z)\in\tilde{\mathfrak{p}}.
\end{equation}

Observe that
$$P(X_{0},\ldots,X_{n})-\sum_{i=1}^{\psi(N)}r_{i}b_{i}(z)\in\tilde{\mathfrak{p}}\cap R_{\alpha}[X_{0},\ldots, X_{n}].$$

Then, let us apply $\text{ev}_{\alpha}(.)$ to \eqref{base}. We get
\begin{align*}
P(X_{0},\ldots,X_{n})-\sum_{i=1}^{\psi(N)}\text{ev}_{\alpha}(r_{i})\text{ev}_{\alpha}(b_{i}(z))&\in\text{ev}_{\alpha}\left(\tilde{\mathfrak{p}}\cap R_{\alpha}[X_{0},\ldots, X_{n}]\right)\\
&=\text{ev}_{\alpha}\left(\tilde{\mathfrak{p}}\cap \overline{\bm{k}}[z][X_{0},\ldots, X_{n}]\right).
\end{align*}

Therefore the family $$\left\{[\text{ev}_{\alpha}(b_{i}(z))]_{\text{ev}_{\alpha}\left(\tilde{\mathfrak{p}}\cap \overline{\bm{k}}[z][X_{0},\ldots, X_{n}]\right)}\right\}_{1\leq i\leq\psi(N)}$$
generates $\mathcal{M}_{\text{ev}_{\alpha}\left(\tilde{\mathfrak{p}}\cap \overline{\bm{k}}[z][X_{0},\ldots, X_{n}]\right)}(N)$ over $\overline{\bm{k}}$. Hence, we obtain \eqref{a_dem} and Proposition \ref{phil_local_ens} is proved.

\end{proof}

We now deduce Corollary \ref{phil_local}.

\begin{proof}[Proof of Corollary \ref{phil_local}]

First, the matrices $A(z)$ and $A^{-1}(z)$ only have finitely many poles. Then, there exists $0<\rho_{0}<1$ such that for all $\alpha\in\overline{\mathbb{K}}$, $0<|\alpha|<\rho_{0}$, $\alpha$ is regular with respect to System \eqref{syst_gen}, and all the $f_{i}(z)$'s are analytic at $\alpha$. Now, in Proposition \ref{phil_local_ens}, take $S$ to be the set of all $\alpha\in\overline{\mathbb{K}}$, such that $0<|\alpha|<\rho_{0}$. Assumption 1 of Proposition \ref{phil_local_ens} is guaranteed by Theorem \ref{nish_gen}, Assumption 2 is satisfied, and Corollary \ref{phil_local} follows from Proposition \ref{phil_local_ens}.	
\end{proof}

\subsection{Proof of the inhomogeneous counterpart of Theorem \ref{phil_p}}
\label{step2}

In this section, we use the same approach as P. Philippon \cite{Phil1} to obtain the following inhomogeneous counterpart of Theorem \ref{phil_p} from Corollary \ref{phil_local}. 

\begin{prop}
	\label{phil_p_inhom}
We continue with the assumptions of Theorem \ref{nish_gen}. Let us assume further that the extension \\$\overline{\mathbb{K}}(z)(f_{1}(z), \ldots, f_{n}(z))$ is regular over $\overline{\mathbb{K}}(z)$. 

Then, for all polynomial $P(X_{1},\ldots, X_{n})\in\overline{\mathbb{K}}[X_{1},\ldots,X_{n}]$ of total degree $N$ in $X_{1},\ldots,X_{n}$ such that 
$$P(f_{1}(\alpha),\ldots, f_{n}(\alpha))=0,$$
there exists a polynomial $Q(z,X_{1},\ldots, X_{n})\in \overline{\mathbb{K}}[z][X_{1},\ldots,X_{n}]$ of total degree $N$ in $X_{1},\ldots,X_{n}$ such that 
$$Q(z,f_{1}(z),\ldots, f_{n}(z))=0,$$
and 
$$Q(\alpha,X_{1},\ldots, X_{n})=P(X_{1},\ldots, X_{n}).$$	
	
\end{prop}

\begin{proof}[Proof of Proposition \ref{phil_p_inhom}]
Let us keep the assumptions of Proposition \ref{phil_p_inhom}. Let $P(X_{1},\ldots, X_{n})\in\overline{\mathbb{K}}[X_{1},\ldots,X_{n}]$ be a polynomial of total degree $N$ in $X_{1},\ldots,X_{n}$ such that 
\begin{equation}
\label{rel_nombre}
P(f_{1}(\alpha),\ldots, f_{n}(\alpha))=0.
\end{equation}

Let us consider $\rho$ from Corollary \ref{phil_local} and $r\in\mathbb{N}$ such that $0<|\alpha^{d^{r}}|<\rho$. We can derive from $d$-Mahler System \eqref{syst_gen} the following equality.

\begin{equation}
\label{lien_mat}
\begin{pmatrix}
f_{1}(z) \\ \vdots \\ f_{n}(z) 
\end{pmatrix} = B(z)\begin{pmatrix}
f_{1}(z^{d^{r}}) \\ \vdots \\ f_{n}(z^{d^{r}})
\end{pmatrix},
\end{equation}
where
$$B(z)=A^{-1}(z)A^{-1}(z^{d})\cdots A^{-1}(z^{d^{r-1}}).$$

As $\alpha$ is regular for System \eqref{syst_gen}, it is neither a pole of $B(z)$ nor a pole of $B^{-1}(z)$. Then, let us set $z=\alpha$ in \eqref{lien_mat}. We obtain

\begin{equation}
\begin{pmatrix}
f_{1}(\alpha) \\ \vdots \\ f_{n}(\alpha) 
\end{pmatrix} =B(\alpha)\begin{pmatrix}
f_{1}(\alpha^{d^{r}}) \\ \vdots \\ f_{n}(\alpha^{d^{r}})
\end{pmatrix}.
\end{equation}
Now, let us set

$$Q(X_{1},\ldots, X_{n})=P(\langle B_{1}(\alpha),\overline{X}\rangle,\ldots, \langle B_{n}(\alpha),\overline{X}\rangle),$$
where, for all $i\in\{1,\ldots,n\}$, $B_{i}(z)$ denotes the $i$-th row of the matrix $B(z)$, $\overline{X}=\begin{pmatrix}
X_{1}\\ \vdots \\ X_{n}
\end{pmatrix}$, and $\langle.,.\rangle$ refers to the classical scalar product on $\overline{\mathbb{K}}\{z\}^{n}$. We get
$$Q\left(f_{1}\left(\alpha^{d^{r}}\right),\ldots, f_{n}\left(\alpha^{d^{r}}\right)\right)=P(f_{1}(\alpha),\ldots, f_{n}(\alpha))=0.$$

As $B(\alpha)$ is invertible, $\deg_{X}(Q)=\deg_{X}(P)=N$. We now apply Corollary \ref{phil_local} to $Q$ and $\alpha^{d^{r}}$. There exists a polynomial $R(z,X_{1},\ldots,X_{n})\in\overline{\mathbb{K}}[z][X_{1},\ldots,X_{n}]$ of degree $N$ in $X_{1},\ldots,X_{n}$ such that :
$$R(z,f_{1}(z),\ldots, f_{n}(z))=0,$$
and 
$$R(\alpha^{d^{r}},X_{1},\ldots, X_{n})=Q(X_{1},\ldots, X_{n}).$$

It follows that
$$R(z^{d^{r}},f_{1}(z^{d^{r}}),\ldots, f_{n}(z^{d^{r}}))=0.$$

Now, let us write $B_{i}^{-1}(z), i=1,\ldots,n$, to denote the $i$-th row of the matrix $B^{-1}(z)$. Let $b(z)\in\mathbb{K}[z]$ be a polynomial such that for all $i\in\{1,\ldots, n\}$, $b(z)B_{i}^{-1}(z)\in \mathbb{K}[z]^{n}$ and for all $k\in\mathbb{N}$, $\alpha^{d^{k}}$ is not a zero of $b(z)$ (which is possible because for all $k\in\mathbb{N}$, $\alpha^{d^{k}}$ is not a pole of $A(z)$). Let us set
$$S(z,X_{1},\ldots, X_{n})=R\left(z^{d^{r}},\langle B_{1}^{-1}(z),\overline{X}\rangle,\ldots, \langle B_{n}^{-1}(z),\overline{X}\rangle\right)\left(\frac{b(z)}{b(\alpha)}\right)^{N}.$$
By construction, we have
$$S(z,X_{1},\ldots, X_{n})\in\overline{\mathbb{K}}[z][X_{1},\ldots, X_{n}].$$
As $B^{-1}(z)$ is invertible, $\deg_{X}(S)=\deg_{X}(R)=N$.
Besides, we obtain
$$S\left(z,f_{1}(z),\ldots, f_{n}(z)\right)=R\left(z^{d^{r}},f_{1}\left(z^{d^{r}}\right),\ldots, f_{n}\left(z^{d^{r}}\right)\right)\left(\frac{b(z)}{b(\alpha)}\right)^{N}=0.$$

Finally, as $\alpha$ is regular, we get
\begin{align*}
S\left(\alpha,X_{1},\ldots X_{n}\right)&=R\left(\alpha^{d^{r}},\langle B_{1}^{-1}(\alpha),\overline{X}\rangle,\ldots, \langle B_{n}^{-1}(\alpha),\overline{X}\rangle\right)\\
&=Q\left(\langle B_{1}^{-1}(\alpha),\overline{X}\rangle,\ldots, \langle B_{n}^{-1}(\alpha),\overline{X}\rangle\right)\\
&=P\left(\left\langle B_{1}(\alpha),\begin{pmatrix}
\langle B_{1}^{-1}(\alpha),\overline{X}\rangle\\ \vdots \\ \langle B_{n}^{-1}(\alpha),\overline{X}\rangle\end{pmatrix}\right\rangle,\ldots, \left\langle B_{n}(\alpha),\begin{pmatrix}
\langle B_{1}^{-1}(\alpha),\overline{X}\rangle \\ \vdots \\ \langle B_{n}^{-1}(\alpha),\overline{X}\rangle\end{pmatrix}\right\rangle\right)\\ 
&=P\left(X_{1},\ldots, X_{n}\right).
\end{align*}

Thus, we found a polynomial $S(z,X_{1},\ldots, X_{n})\in\overline{\mathbb{K}}[z][X_{1},\ldots,X_{n}]$ of degree $N$ in $X_{1},\ldots, X_{n}$ such that :
$$S(z,f_{1}(z),\ldots, f_{n}(z))=0,$$
and
\begin{equation}
S(\alpha,X_{1},\ldots, X_{n})=P(X_{1},\ldots, X_{n}).
\end{equation} 

The inhomogeneous counterpart of Theorem \ref{phil_p} is proved.
\end{proof}

\subsection{End of the proof of Theorem \ref{phil_p}}
\label{hom}

The first part of the proof of Theorem \ref{phil_p} consists in showing the following analogue of \cite[Théorème 4.1]{A-F}.

\begin{thm}
	\label{thm_phil_lin}
	Under the assumptions of Theorem \ref{phil_p}, we have
	\begin{equation}
	\text{Rel}_{\overline{\mathbb{K}}}(f_{1}(\alpha),\ldots,f_{n}(\alpha))=\text{ev}_{\alpha}(\text{Rel}_{\overline{\mathbb{K}}(z)}(f_{1}(z),\ldots,f_{n}(z))),
	\end{equation}
	where for a field $L$ and elements $u_{1},\ldots,u_{n}$ of a $L$-vector space,
	$\text{Rel}_{L}(u_{1},\ldots,u_{n})$ denotes the set of linear relations over $L$ between the $u_{i}$'s.
\end{thm}

We do not reproduce the proof of Theorem \ref{thm_phil_lin}. It can be proved as in \cite{A-F}, by induction on the dimension of $\text{Rel}_{\overline{\mathbb{K}}(z)}(f_{1}(z),\ldots,f_{n}(z))$. However, we give here more details about how to deduce Theorem \ref{phil_p} from Theorem \ref{thm_phil_lin} (see also \cite{A-F_var}). 
Let $P(X_{1},\ldots, X_{n})\in\overline{\mathbb{K}}[X_{1},\ldots,X_{n}]$ be homogeneous of degree $N$ in $X_{1},\ldots,X_{n}$ such that  
\begin{equation}
\label{rel_alg}
P(f_{1}(\alpha),\ldots, f_{n}(\alpha))=0.
\end{equation}

Let $\mathcal{G}_{N}$ denote the set of all monic monomials of degree $N$ in $f_{1}(z),\ldots, f_{n}(z)$. Then, \eqref{rel_alg} can be seen as a linear relation over $\overline{\mathbb{K}}$ between specializations at $z=\alpha$ of elements of $\mathcal{G}_{N}$. Our aim is to show that the elements of $\mathcal{G}_{N}$ satisfy a $d$-Mahler system for which $\alpha$ is still regular and apply Theorem \ref{thm_phil_lin} to the functions of $\mathcal{G}_{N}$ and $P$.

To do so, in the sequel, we define by induction on $N$, $n$ vectors $M_{N}^{1}(z),\ldots, M_{N}^{n}(z)$ which satisfy the following properties.
\begin{enumerate}
	\item 
	For all $i\in\{1,\ldots,n\}$, $M_{N}^{i}(z)$ is composed of $n^{N-1}$ rows. We write
	$$M_{N}^{i}(z)=\begin{pmatrix}
    L_{N,1}^{i}(z)\\
    \colon\\
    L_{N,n^{N-1}}^{i}(z)
	\end{pmatrix}.$$
	\item
	For all $i\in\{1,\ldots,n\}$, $j\in\{1,\ldots,n^{N-1}\}$, $L_{N,j}^{i}(z)\in\mathcal{G}_{N}$.
	\item
	$\mathcal{G}_{N}\subseteq \{L_{N,j}^{i}(z)\}_{i,j}$.
\end{enumerate}

Let us set
\begin{equation}
\label{notation_MN}
M_{N}(z)=\begin{pmatrix}
M_{N}^{1}(z)\\
\colon\\
M_{N}^{n}(z)
\end{pmatrix}.
\end{equation}

This is a vector of $n^{N}$ rows of elements of $\mathcal{G}_{N}$. Let us define \eqref{notation_MN} by induction on $N$ in the following way.
\begin{enumerate}
	\item[(a)] 
	For all $i\in\{1,\ldots,n\}$,
	$$M_{1}^{i}(z)=f_{i}(z).$$
	\item[(b)]
	For all $N\geq 2$, $i\in\{1,\ldots,n\}$,
	$$M_{N}^{i}(z)=M_{N-1}(z)f_{i}(z)=\begin{pmatrix}
	M_{N-1}^{1}(z)f_{i}(z)\\
	\colon\\
	M_{N-1}^{n}(z)f_{i}(z)
	\end{pmatrix}.$$
\end{enumerate}

We see that this definition allows $M_{N}(z)$ to satisfy properties 1-3, for all $N\geq 1$.

Now, we have the following Lemma.
\begin{lem}
	\label{mahl_kronecker}
	The elements of $\mathcal{G}_{N}$ satisfy the following $d$-Mahler system
\begin{equation}
\label{syst_kronecker}
M_{N}(z^{d})=A^{\otimes N}(z)M_{N}(z),
\end{equation}	
where $\otimes$ stands for the Kronecker product.
\end{lem}

\begin{proof}[Proof of Lemma \ref{mahl_kronecker}]
	We prove Lemma \ref{mahl_kronecker} by induction on $N$. For $N=1$, \eqref{syst_kronecker} holds true. Now, let us assume that \eqref{syst_kronecker} is satisfied at the rank $N-1$. Let us set $A(z)=(a_{i,j}(z))_{i,j}$ for the matrix of $d$-Mahler System \eqref{syst_gen}. Then, we have for all $i\in\{1,\ldots,n\}$
	
	\begin{align}
	M_{N}^{i}(z^{d})&=M_{N-1}(z^{d})f_{i}(z^{d})\nonumber\\
	&=A^{\otimes N-1}(z)M_{N-1}(z)f_{i}(z^{d}), \text{ by assumption}\nonumber\\
	&=A^{\otimes N-1}(z)M_{N-1}(z)\sum_{j=1}^{n}a_{i,j}(z)f_{j}(z)\nonumber\\
	\label{partie_kronecker}
	&=\sum_{j=1}^{n}\left(a_{i,j}(z)A^{\otimes N-1}(z)\right)M_{N}^{j}(z).
	\end{align}
If we cut the rows of the matrix $A^{\otimes N}(z)$ from top to bottom into $n$ blocks of $n^{N-1}$ rows, \eqref{partie_kronecker} corresponds to the product of the $i$-th block of $A^{\otimes N}(z)$ by $M_{N}(z)$. This implies Lemma \ref{mahl_kronecker}.	
\end{proof}

We are now able to end the proof of Theorem \ref{phil_p}

\begin{proof}[End of the proof of Theorem \ref{phil_p}]
	
By property of Kronecker product (see for example \cite{Graham}), the coefficients of $A^{\otimes N}(z)$ are products of elements of $A(z)$ and we have 
$$\det\left(A^{\otimes N}(z)\right)=\det\left(A(z)\right)^{nN}.$$
We deduce that $\alpha$ is still a regular number for $d$-Mahler System \eqref{syst_kronecker}. On the other hand, $\overline{\mathbb{K}}(z)(f_{1}(z),\ldots, f_{n}(z))$ is regular over $\overline{\mathbb{K}}(z)$ and $$\overline{\mathbb{K}}(z)(\mathcal{G}_{N})\subseteq \overline{\mathbb{K}}(z)(f_{1}(z),\ldots, f_{n}(z)).$$
Then, by \cite[Corollary A1.6]{Eisenbud}, $\overline{\mathbb{K}}(z)(\mathcal{G}_{N})$ is separable over $\overline{\mathbb{K}}(z)$. It follows that $\overline{\mathbb{K}}(z)(\mathcal{G}_{N})$ is regular over $\overline{\mathbb{K}}(z)$. Hence, Theorem \ref{phil_p} follows from Theorem \ref{thm_phil_lin}.	
\end{proof}

\medskip

We end this section with the following remark, which allows us to consider Theorem \ref{phil_p} from an other point of view.

\begin{rem}
	\label{equi_ccl_prime}
	\em{
		Let us keep the assumptions of Theorem \ref{nish_gen}. Then, the regularity of $\overline{\mathbb{K}}(z)(f_{1}(z),\ldots,f_{n}(z))$ in the assumptions of Theorem \ref{phil_p} can be replaced by
		\begin{enumerate}
			\item [(i)]
			$\text{ev}_{\alpha}\left(\tilde{\mathfrak{p}}\cap \overline{\mathbb{K}}[z][X_{0},\ldots, X_{n}]\right)$ is prime in $\overline{\mathbb{K}}[X_{0},\ldots, X_{n}]$.	
		\end{enumerate}
		
		Indeed, if (i) is satisfied, we can reproduce the proof of Proposition \ref{phil_local_ens} from \eqref{fin_en_sachant_premier} to the end to show that Proposition \ref{phil_local_ens} holds true. Then, Corollary \ref{phil_local} holds true and it follows from Sections \ref{step2} and \ref{hom} that Theorem \ref{phil_p} holds true. Reciprocally, if Theorem \ref{phil_p} holds true, we have $\text{ev}_{\alpha}(\tilde{\mathfrak{p}}\cap \overline{\mathbb{K}}[z][X_{0},\ldots, X_{n}])=\tilde{\mathfrak{p}}_{\alpha}$, and (i) is satisfied.} 
\end{rem}

\section{Proof of Theorem \ref{thm_faux}}
\label{faux}
Let us keep the notations and assumptions of Theorem \ref{thm_faux}. We recall that $R_{\alpha}$ is the localization of the ring $\overline{\bm{k}}[z]$ at the ideal $(z-\alpha)$. Before going through the proof of Theorem \ref{thm_faux}, let us make a remark in the vein of Remark \ref{alpha_base_mod} and recall basic facts about Cartier operators, along with a result of S. Mac Lane concerning separability. Let $N\in\mathbb{N}$. We set
$$\mathcal{G}=\text{vect}_{\overline{\bm{k}}(z)}\{Q(f_{1}(z),\ldots,f_{n}(z)), Q\in\overline{\bm{k}}(z)[X_{1},\ldots, X_{n}], \text{ homogeneous }, \deg_{X}(Q)\leq N\}$$ 

\begin{rem}
	\label{alpha_base_poly}
	\em{
	We can prove, in the same way as in the proof of \cite[Lemma 3]{Nest_shid}, that there exist monic monomials $M_{l}(X_{1},\ldots, X_{n})$, with $\deg_{X}(M_{l})\leq N$, $l=1,\ldots,s$, such that the family $\{M_{l}(f_{1}(z),\ldots, f_{n}(z))\}_{l}$ is a basis of $\mathcal{G}$ which satisfies the following property.

\begin{enumerate}
	\item [$\left(*_{\alpha}\right)$]
	For all $P(X_{1},\ldots,X_{n})\in R_{\alpha}[X_{1},\ldots,X_{n}]$, there exist $P_{1}(z),\ldots, P_{s}(z)\in R_{\alpha}$ such that
	$$P(f_{1}(z),\ldots,f_{n}(z))=\sum_{l=1}^{s} P_{l}(z)M_{l}(f_{1}(z),\ldots, f_{n}(z)).$$
\end{enumerate}}
\end{rem}

If $\bm{k}$ has characteristic $p$, we recall some basic facts about Cartier operators. Let $f(z)=\sum_{n=0}^{+\infty}a(n)z^{n}\in \tilde{\bm{k}}[[z]]$. Let $r\in\{0,\ldots,p-1\}$. The $r$-th Cartier operator over $\tilde{\bm{k}}[[z]]$ is defined by
$$\Lambda_{r}(f)=\sum_{n=0}^{+\infty}a(np+r)^{1/p}z^{n}.$$

Then, we recall the following result.

\begin{prop}
	\label{Cartier}
	Let $f,g\in \tilde{\bm{k}}[[z]]$.
	\begin{enumerate}
		\item 
		We have $$f(z)=\sum_{i=0}^{p-1}\Lambda_{i}(f)^{p}z^{i}.$$
		In particular, 
		$$f(z)\neq 0\Rightarrow\exists i\in\{0,\ldots,p-1\},\Lambda_{i}(f)\neq 0.$$
		\item
		Let $i\in\{0,\ldots,p-1\}$. Then
		$$\Lambda_{i}(fg^{p})=\Lambda_{i}(f)g.$$ 
	\end{enumerate}
	
\end{prop}

Besides, if $\bm{k}$ has characteristic $p$, we write $\bm{k}^{1/p^{\infty}}$ to denote the perfect closure of $\bm{k}$. That is, the union over $n$ of the fields generated by the $p^{n}$-th roots of all the elements of $\bm{k}$. Finally, we recall a fundamental theorem from S. Mac Lane \cite{Maclane} (see also \cite[Theorem A1.3]{Eisenbud}).

\begin{thm}[S. Mac Lane]
	\label{thm_sep_fonda}
	Let $\bm{k}\subseteq L$ be a field extension. Then, this extension is separable if and only if every family $\{x_{i}\}_{i}$ of elements of $L$ that is linearly independent over $\bm{k}$ remains linearly independent over $\bm{k}^{1/p^{\infty}}$.  
\end{thm}

Then, we prove the following result.

\begin{prop}
	\label{tjrs_sep}
	Let $\bm{k}$ be a valued field and let us assume that $f_{1}(z),\ldots, f_{n}(z)\in\overline{\bm{k}}\{z\}$ are analytic functions on a domain $\mathcal{U}\subseteq\tilde{\bm{k}}$ which contains the origin. Then, the extension $\overline{\bm{k}}(z)(f_{1}(z), \ldots, f_{n}(z))$ is separable over $\overline{\bm{k}}(z)$. 
\end{prop}

\begin{proof}[Proof of Proposition \ref{tjrs_sep}]
	
	If the characteristic of $\bm{k}$ is zero, the result is known. Now, let us assume that $\bm{k}$ has characteristic $p>0$.	Let us assume by contradiction that  $\overline{\bm{k}}(z)(f_{1}(z), \ldots, f_{n}(z))$ is not separable over $\overline{\bm{k}}(z)$. Let us note that 
	$$\overline{\bm{k}}(z)^{1/p^{\infty}}=\cup_{k=0}^{+\infty}\overline{\bm{k}}\left(z^{1/p^{k}}\right).$$  
	Besides, let us set
	$$\overline{\bm{k}}[z]^{1/p^{\infty}}=\cup_{k=0}^{+\infty}\overline{\bm{k}}\left[z^{1/p^{k}}\right].$$
	
	By Theorem \ref{thm_sep_fonda}, there exist elements $g_{1}(z),\ldots,g_{m}(z)\in\overline{\bm{k}}(z)(f_{1}(z), \ldots, f_{n}(z))$ which are linearly independent over $\overline{\bm{k}}(z)$ but linearly dependent over $\overline{\bm{k}}(z)^{1/p^{\infty}}$. Let $D\in\overline{\bm{k}}[z][f_{1}(z), \ldots, f_{n}(z)]\setminus\{0\}$ be such that $$g_{i}(z)D\in\overline{\bm{k}}[z][f_{1}(z), \ldots, f_{n}(z)], \text{ }\forall 1\leq i\leq m.$$
	Then, $g_{1}(z)D,\ldots,g_{m}(z)D$ are linearly independent over $\overline{\bm{k}}(z)$ but linearly dependent over $\overline{\bm{k}}(z)^{1/p^{\infty}}$. Hence, even if it means replacing each $g_{i}(z)$ by $g_{i}(z)D$, we assume that for all $1\leq i\leq m$, $g_{i}(z)\in\overline{\bm{k}}[z][f_{1}(z), \ldots, f_{n}(z)]$.
	
    Then, there exist elements $G_{1}(z),\ldots, G_{m}(z)\in \overline{\bm{k}}[z]^{1/p^{\infty}}$ not all zero such that 
	$$\sum_{i=1}^{m}G_{i}(z)g_{i}(z)=0.$$
	
	Without any loss of generality, we may assume that $G_{m}(z)\neq 0$. On the other hand, there exists an integer $\mu\geq 1$ such that $G_{i}(z)^{p^{\mu}}\in\overline{\bm{k}}[z]$ for all $1\leq i\leq m$. Then, we have
    \begin{equation}
    \label{pmin_bis}
    \sum_{i=1}^{m}G_{i}(z)^{p^{\mu}}g_{i}(z)^{p^{\mu}}=0.
    \end{equation} 
   
   Let us note that for all $i\in\{1,\ldots,m\}$, $$G_{i}(z)^{p^{\mu}},g_{i}(z)\in\overline{\bm{k}}[z][f_{1}(z), \ldots, f_{n}(z)]\subseteq\overline{\bm{k}}\{z\}.$$ 
   Now, let us choose for every integer $j\in\{1,\ldots, \mu\}$ a Cartier operator $\Lambda^{(j)}$ such that $$\Lambda^{(\mu)}\circ\cdots\circ\Lambda^{(1)}(G_{m}(z)^{p^{\mu}})\neq 0.$$

	We apply $\Lambda:=\Lambda^{(\mu)}\circ\cdots\circ\Lambda^{(1)}$ to \eqref{pmin_bis} and get
	\begin{equation}
	\label{eq_non_triv}
	\sum_{i=1}^{m}\Lambda(G_{i}(z)^{p^{\mu}})g_{i}(z)=0.
	\end{equation} 
	
Then, \eqref{eq_non_triv} is a non-trivial linear relation between the $g_{i}(z)$'s over $\overline{\bm{k}}(z)$ and a contradiction. Proposition \ref{tjrs_sep} is thus proved.

\end{proof}

Before proving Theorem \ref{thm_faux}, we introduce some definitions. 
We recall that $\bm{k}$ denote a valued field, $\bm{k}_{c}$ its completion. Its valuation extends uniquely to $\overline{\bm{k}_{c}}$, and $\tilde{\bm{k}}$ denote the completion of $\overline{\bm{k}_{c}}$ with respect to this valuation. Now, let $\mathcal{U}\subseteq\tilde{\bm{k}}$ be a domain. We say that a function is meromorphic on $\mathcal{U}$ if there exists a (possibly empty) discrete closed subset $\mathcal{P}$ of $\mathcal{U}$ such that $f(z)$ is analytic on $\mathcal{U}\setminus\mathcal{P}$, and each element of $\mathcal{P}$ is a pole of $f(z)$. Then, for all $\alpha\in\mathcal{P}$, $f(z)$ admits a convergent Laurent power series expansion in a punctured neighbourhood of $\alpha$ with coefficients in $\tilde{\bm{k}}$, of the form $\sum_{n=-N}^{+\infty}a_{n}(z-\alpha)^{n}$. We notice that if $\{f_{i}(z)\}_{1\leq i\leq n}\subset\mathbb{K}\{z\}$ satisfies System \eqref{syst_gen}, if $0<|\alpha|<1$, and if for all $k\in\mathbb{N}$ the number $\alpha^{d^{k}}$ is not a pole of $A^{-1}(z)$, then the $f_{i}(z)$ are well-defined at $\alpha$ and $\{f_{i}(z)\}_{1\leq i\leq n}\subset C\{z-\alpha\}$. 
\medskip

We are now able to prove Theorem \ref{thm_faux}.

\begin{proof}[Proof of Theorem \ref{thm_faux}]
Let us assume that the extension $\overline{\bm{k}}(z)(f_{1}(z), \ldots, f_{n}(z))$ is not regular over $\overline{\bm{k}}(z)$. We recall that $\overline{\bm{k}}(z)(f_{1}(z), \ldots, f_{n}(z))$ is regular over $\overline{\bm{k}}(z)$ if
\begin{enumerate}
	\item
	$\overline{\bm{k}}(z)(f_{1}(z), \ldots, f_{n}(z))$ is separable over $\overline{\bm{k}}(z)$
	\item
	Every element of $\overline{\bm{k}}(z)(f_{1}(z), \ldots, f_{n}(z))$ that is algebraic over $\overline{\bm{k}}(z)$ belongs to $\overline{\bm{k}}(z)$.
\end{enumerate}

By Proposition \ref{tjrs_sep}, we only have to prove that the conclusion of Theorem \ref{thm_faux} holds true when there exists an element of $\overline{\bm{k}}(z)(f_{1}(z), \ldots, f_{n}(z))$ that is algebraic over $\overline{\bm{k}}(z)$ but does not belong to $\overline{\bm{k}}(z)$.

Thus, let us assume that there exists an element $a(z)\in\overline{\bm{k}}(z)(f_{1}(z),\ldots,f_{n}(z))\cap \overline{\bm{k}(z)}\setminus\overline{\bm{k}}(z)$. We can write
\begin{equation}
\label{frac1}
a(z)=\frac{P(f_{1}(z),\ldots,f_{n}(z))}{Q(f_{1}(z),\ldots,f_{n}(z))},
\end{equation}
where $P(X_{1},\ldots,X_{n}), Q(X_{1},\ldots,X_{n})\in \overline{\bm{k}}[z][X_{1},\ldots,X_{n}]$ are polynomials of total degree less than or equal to some integer $N\geq 0$. 

We recall that $\mathcal{G}$ denotes the $\overline{\bm{k}}(z)$-vector space generated by all homogeneous polynomials of degree less than or equal to $N$ in $f_{1}(z),\ldots,f_{n}(z)$. 

By Remark \ref{alpha_base_poly}, there exist monic monomials $M_{l}(X_{1},\cdots,X_{n})$, with $\deg_{X}(M_{l})\leq N$, $l=1,\ldots,s$, such that the family $\{M_{l}(\{f_{i}(z)\})\}_{l}$ is a basis of $\mathcal{G}$ over $\overline{\bm{k}}(z)$ which satisfies Property $\left(*_{\alpha}\right)$ of Remark \ref{alpha_base_poly}. Then, \eqref{frac1} turns into
\begin{equation}
\label{frac2}
a(z)=\frac{N_{1}(z)M_{1}(\{f_{i}(z)\})+\cdots+ N_{s}(z)M_{s}(\{f_{i}(z)\})}{D_{1}(z)M_{1}(\{f_{i}(z)\})+\cdots+D_{s}(z)M_{s}(\{f_{i}(z)\})},
\end{equation}
where for all $l\in\{1,\ldots,s\}$, $N_{l}(z),D_{l}(z)\in \overline{\bm{k}}[z]$.
 
We can rewrite \eqref{frac2} in the following way

\begin{equation}
\label{relfct}
F_{1}(z)M_{1}(\{f_{i}(z)\})+\cdots+F_{s}(z)M_{s}(\{f_{i}(z)\})=0,
\end{equation}
where for all $l\in\{1,\ldots,s\}$, $F_{l}(z)=D_{l}(z)a(z)-N_{l}(z)$.

We may assume without any loss of generality that for all $l$, $F_{l}(z)\in \tilde{\bm{k}}\{z-\alpha\}$. Indeed, on the one hand, as the functions $f_{1}(z),\ldots,f_{n}(z)\in \tilde{\bm{k}}\{z-\alpha\}$, $a(z)$ can be expressed as a Laurent power series at the point $z=\alpha$. If $a(z)\notin \tilde{\bm{k}}\{z-\alpha\}$, writing $u>0$ the order of the pole of $a(z)$ at $z=\alpha$, we could replace $a(z)$ by the function
$$(z-\alpha)^{u}a(z)\in\overline{\bm{k}}(z)(f_{1}(z),\ldots,f_{n}(z))\cap \overline{\bm{k}(z)}\setminus\overline{\bm{k}}(z),$$
which has no pole at $z=\alpha$. Therefore, we can assume that $a(z)\in \tilde{\bm{k}}\{z-\alpha\}$.
Then, as for all $l\in\{1,\ldots,s\}$, $N_{l}(z),D_{l}(z)\in \overline{\bm{k}}[z]$, we get that $F_{l}(z)\in \tilde{\bm{k}}\{z-\alpha\}$. 

Now, let us notice that we can assume without any loss of generality that
\begin{equation}
\label{nonnul}
\exists l_{0}\in\{1,\ldots,s\}, F_{l_{0}}(\alpha)\neq 0.
\end{equation}

Indeed, $F_{l}(z)\in \tilde{\bm{k}}\{z-\alpha\}$. Therefore, if \eqref{nonnul} is not satisfied, let $v>0$ denote the minimal order at $\alpha$ as a zero of the functions $F_{l}(z)$. Then, instead of \eqref{relfct}, we could consider the following equation
 
\begin{equation}
\label{relfctbis}
G_{1}(z)M_{1}(\{f_{i}(z)\})+\cdots+G_{s}(z)M_{s}(\{f_{i}(z)\})=0,
\end{equation}
where for all $l\in\{1,\ldots,s\}$, $G_{l}(z)=\frac{F_{l}(z)}{(z-\alpha)^{v}}\in \tilde{\bm{k}}\{z-\alpha\}$. The functions $G_{l}$ satisfy \eqref{nonnul}. Hence, even if it means replacing \eqref{nonnul} by \eqref{relfctbis}, we assume that \eqref{nonnul} holds. 

Then, we have
\begin{equation}
\label{rel}
F_{1}(\alpha)M_{1}(\{f_{i}(\alpha)\})+\cdots+F_{s}(\alpha)M_{s}(\{f_{i}(\alpha)\})=0.
\end{equation}
Hence, setting
$$P\left(X_{1},\ldots, X_{n}\right)=\sum_{l=1}^{s}F_{l}(\alpha)M_{l}\left(X_{1},\ldots, X_{n}\right),$$
we get
\begin{equation}
\label{rel_contr}
P(f_{1}(\alpha),\ldots, f_{n}(\alpha))=0.
\end{equation}
Let us assume by contradiction that the relation \eqref{rel_contr} lifts into a functional relation over $\overline{\bm{k}}(z)$. Let $N'\leq N$ denote the total degree of $P\left(X_{1},\ldots, X_{n}\right)$. Then, there exists a polynomial\\ $Q\left(z,X_{1},\ldots, X_{n}\right)\in \overline{\bm{k}}[z][X_{1},\ldots, X_{n}]$ of total degree $N'$ in $X_{1},\ldots, X_{n}$ such that
\begin{equation}
\label{annul}
Q\left(z,f_{1}(z),\ldots, f_{n}(z)\right)=0,
\end{equation}

and
\begin{equation}
\label{eval}
Q\left(\alpha,X_{1},\ldots, X_{n}\right)=P\left(X_{1},\ldots, X_{n}\right).
\end{equation}

Let us notice that the family $\{M_{l}(X_{1},\ldots, X_{n})\}_{l}$ is free over $\overline{\bm{k}}(z)$. Let $\{N_{j}(X_{1},\ldots, X_{n})\}_{1\leq j\leq t}$ be a family of monic monomials such that the family $\{M_{l}(X_{1},\ldots, X_{n}),N_{j}(X_{1},\ldots, X_{n})\}_{l,j}$ is a basis of the $\overline{\bm{k}}(z)$-vector space spanned by all homogeneous polynomials of degree less than or equal to $N$ in $X_{1},\ldots, X_{n}$.

Then, we can write the polynomial $Q\left(z,X_{1},\ldots, X_{n}\right)$ in the following way.
$$Q\left(z,X_{1},\ldots, X_{n}\right)=\sum_{l=1}^{s}Q_{l}(z)M_{l}\left(X_{1},\ldots, X_{n}\right)+\sum_{j=1}^{t}R_{j}(z)N_{j}\left(X_{1},\ldots, X_{n}\right),$$
where for all $l$ and $j$, $Q_{l}(z), R_{j}(z)\in \overline{\bm{k}}[z]$. By \eqref{eval}, we have
\begin{equation}
\label{eval_nulle}
\begin{cases}
Q_{l}(\alpha)&=F_{l}(\alpha), \text{ } \forall 1\leq l\leq s\\
R_{j}(\alpha)&=0, \text{ } \forall 1\leq j\leq t.
\end{cases}
\end{equation}
Now, we have
\begin{align}
0 &=Q\left(z,f_{1}(z),\ldots, f_{n}(z)\right)\nonumber\\
\label{annul_poly_fct}
&=\sum_{l=1}^{s}Q_{l}(z)M_{l}\left(\{f_{i}(z)\}\right)+\sum_{j=1}^{t}R_{j}(z)N_{j}\left(\{f_{i}(z)\}\right).
\end{align}

Let us remark that
$$N_{j}\left(X_{1},\ldots, X_{n}\right)\in \overline{\bm{k}}[X_{1},\ldots, X_{n}]\subseteq R_{\alpha}[X_{1},\ldots, X_{n}].$$
Hence, by Remark \ref{alpha_base_poly}, we get that for all $j\in\{1,\ldots,t\}$ there exist polynomials $S_{j,1}(z),\ldots, S_{j,s}(z)\in R_{\alpha}$ such that
$$N_{j}\left(\{f_{i}(z)\}\right)=\sum_{l=1}^{s}S_{j,l}(z)M_{l}\left(\{f_{i}(z)\}\right).$$
Therefore, \eqref{annul_poly_fct} turns into
$$0=\sum_{l=1}^{s}Q_{l}(z)M_{l}\left(\{f_{i}(z)\}\right)+\sum_{j=1}^{t}R_{j}(z)\sum_{l=1}^{s}S_{j,l}(z)M_{l}\left(\{f_{i}(z)\}\right).$$
Now, $\{M_{l}\left(f_{1}(z),\ldots,f_{n}(z)\right)\}_{l}$ is a $\overline{\bm{k}}(z)$-basis of $\mathcal{G}$. Thus, we obtain
$$Q_{l}(z)+\sum_{j=1}^{t}R_{j}(z)S_{j,l}(z)=0, \text{ }\forall l\in\{1,\ldots, s\},$$
and
$$Q_{l}(\alpha)+\sum_{j=1}^{t}R_{j}(\alpha)S_{j,l}(\alpha)=0, \text{ }\forall l\in\{1,\ldots, s\}.$$
By \eqref{eval_nulle}, we get
$$F_{l}(\alpha)=0, \text{ } \forall 1\leq l\leq s,$$
which contradicts \eqref{nonnul}. Theorem \ref{thm_faux} is proved.

\end{proof}

\section{Proof of Theorem \ref{l_mahler_bis} and Corollary \ref{l_mahler}}
\label{sec_courbe}

We prove Theorem \ref{l_mahler_bis} following the strategy of J. Roques \cite{JR}. We extend Proposition 4 and Corollary 5 of \cite{JR} for the base field $C$ instead of $\overline{\mathbb{F}_{p}}$. The analogue of Proposition 4 is the following.

\begin{prop}
	\label{l_prop}
	
	Let $L$ be a finite extension of $C(z)$. Let $d\geq 2$ be an integer such that $p \nmid d$. Let us assume that the endomorphism $\phi_{d}$ of $C(z)$ defined by $\phi_{d}(P(z))=P(z^{d})$ extends to a field endomorphism of $L$. Then, there exist a positive integer $N$ and $z_{N}\in L$ such that
	\begin{enumerate}
		\item[(i)]
		$z_{N}^{N}=z$
		\item[(ii)]
		$L$ is a purely inseparable extension of $C(z_{N})$.
	\end{enumerate}
\end{prop} 
	
	\begin{proof}[Proof of Proposition \ref{l_prop}]
		We still write $\phi_{d}$ to denote its extension to $L$. Let $E$ denote the separable closure of $C(z)$ in $L$. Then, we see that for all $x\in E$, $\phi_{d}(x)\in E$. Hence, $\phi_{d}$ induces a field endomorphism of $E$. Let $X$ denote a smooth projective curve whose function field is $E$ (see for example \cite[I.6]{Hart}). Let $j : \mathbb{P}^{1}(C) \rightarrow \mathbb{P}^{1}(C)$ be the morphism of curves associated with $\phi_{d} : C(z)\rightarrow C(z)$, $f : X\rightarrow X$ the morphism of curves associated with the extension of $\phi_{d}$ to $E$, and $\varphi : X \rightarrow \mathbb{P}^{1}(C)$ the morphism of curves associated with the inclusion $i : C(z) \hookrightarrow E$.
	Then, we have the following commutative diagram.
			
		 \begin{figure}[H]
		 
		 	\begin{center}
		 		\begin{tikzpicture} 
		 	
		 		\node(X) at (0,2) {$X$}; 
		 		\node(X') at (2,2) {$X$}; 
		 		
		 		\node(p1) at (0,0) {$\mathbb{P}^{1}(C)$};
		 		\node(p1') at (2,0) {$\mathbb{P}^{1}(C)$};
		 		\node(g1)[scale=0.8] at (0,-0.6) {$(x_{1},x_{2})$};
		 		\node(g2)[scale=0.8] at (2.25,-0.6) {$\left(1,\left(\frac{x_{2}}{x_{1}}\right)^{d}\right)$};

		 		\node(j1)[scale=0.7] at (-1.5,0) {$\left(1,\left(\frac{x_{2}}{x_{1}}\right)\right)$};
		 		\node(j2)[scale=0.7] at (-1.5,2) {$(x_{1},x_{2})$};

		 		\node(j1')[scale=0.7] at (3.5,0) {$\left(1,\left(\frac{x_{2}}{x_{1}}\right)\right)$};
		 		\node(j2')[scale=0.7] at (3.5,2) {$(x_{1},x_{2})$};
		 		
		 		\node(j)[scale=0.7] at (-1.3,1) {$\varphi$};
		 		\node(j')[scale=0.7] at (3.3,1) {$\varphi$};
		 		\node(f)[scale=0.7] at (1,2.3) {$f$};
		 		\node(g3)[scale=0.7] at (1,-0.8) {$j$};
		 		
		 		\node(D) at (5,1) {(D)};
		 		\draw[->] (g1) -- (g2);
		 		\draw[->] (X) -- (X');
		 		\draw[->] (p1) -- (p1');
		 		\draw[->] (X) -- (p1);
		 		\draw[->] (X') -- (p1');
		 		\draw[->] (j2) -- (j1);
		 		\draw[->] (j2') -- (j1');
		 		
		 		\end{tikzpicture}
		 	\end{center}
		 \end{figure}
		
		Now, we prove that $f$ satisfies the following properties.
		\begin{enumerate}
			\item 
			$f$ is a separable morphism, that is $E/\phi_{d}(E)$ is a separable extension.
		\item
			$f$ has degree $d$.
			\item
			$f$ is totally ramified above any point of $\varphi^{-1}(0)\cup \varphi^{-1}(\infty)$.
		\end{enumerate} 
		
	To prove the first assertion, it suffices to show that $E/\phi_{d}(C(z))$ is separable. But $\phi_{d}(C(z))=C(z^{d})$. As $p\nmid d$, $C(z)/C(z^{d})$ is separable. Besides, by definition, $E/C(z)$ is separable. Assertion 1 follows. The second assertion can be read on the diagram (D). We get $\deg(f)\deg(\varphi)=\deg(\varphi)\deg(j)$. But $\deg(j)=[C(z):\phi_{d}(C(z))]=d$. Finally, let us prove the last assertion. By diagram (D), we get $f^{-1}(\varphi^{-1}(0))=\varphi^{-1}(0)$. Besides, as $X$ is a smooth projective curve, by \cite[II, 6.7, 6.8, Exercise 3.5]{Hart}, the set $\varphi^{-1}(0)$ is finite. Let $x\in \varphi^{-1}(0)$. We deduce that the set $f^{-1}(x)$ has exactly one element. Now, it follows from \cite[II, Proposition 2.6]{Silverman} that $f$ is totally ramified above $x$. The same arguments hold for $\varphi^{-1}(\infty)$ and Assertion 3 is proved. Now, let $g$ be the genus of $X$. We prove that $g\in\{0,1\}$. First, let us recall the Hurwitz formula (see for example \cite[IV.2.4]{Hart}). If $\vartheta : W\rightarrow W$ is a finite separable morphism of curves, we have 
	\begin{equation}
	\label{hurwitz}
	-2(g(W)-1)(n(\vartheta)-1)\geq \sum_{P\in W}(e_{P}-1),
	\end{equation}
	
	\noindent where the integer $g\geq 0$ is the genus of the curve $W$, the integer $n(\vartheta)\geq 1$ is the degree of $\vartheta$ and the integer $e_{P}\geq 1$ is the ramification index of $\vartheta$ at $P$. Now, if $g\notin\{0,1\}$, it follows from Hurwitz formula \eqref{hurwitz} that all the compositions $f^{i}(z), i\geq 0$, are automorphisms of the smooth projective curve $X$. But H. L. Schmid proved \cite{Schmid} that there only exist finitely many automorphisms of $X$, when $g\geq 2$ (see also \cite{Singh}). As $\phi_{d}$ has infinite order, it is the same for $f(z)$ and we get a contradiction. Hence, $g\in\{0,1\}$. But if $g=1$, it follows from Hurwitz formula \eqref{hurwitz} that $f$ is unramified everywhere. This contradicts Assertion 3. Hence, $g=0$. Now, our goal is to prove the following lemma.
	
	\begin{lem}
		\label{diag_com}
		There exists a transcendental element $u$ over $C$ such that $E=C(u)$. Moreover, there exists $P(u)\in C(u)$ such that the following diagram commutes.
		\begin{figure}[H]
			\begin{center}
				\begin{tikzpicture} 
				
				\node(E) at (0,2) {$C(u)$}; 
				\node(E') at (2,2) {$C(u)$}; 
				
				\node(C) at (0,0) {$C(z)$};
				\node(C') at (2,0) {$C(z)$};
				\node(phi1)[scale=0.8] at (0.25,-0.5) {$z^{d}$};
				\node(phi2)[scale=0.8] at (1.7,-0.5) {$z$};
				
				\node(v1)[scale=0.8] at (1.8,2.5) {$u$};
				\node(v2)[scale=0.8] at (0.33,2.5) {$u^{d}$};
				
				\node(j1)[scale=0.8] at (-1,2) {$P(u)$};
				\node(j2)[scale=0.8] at (-1,0) {$z$};
				\node(j1')[scale=0.8] at (3,0) {$z$};
				\node(j2')[scale=0.8] at (3,2) {$P(u)$};
				
				\node(j3)[scale=0.8] at (-0.7,1) {$h_{1}$};
				\node(j3')[scale=0.8] at (2.7,1) {$h_{1}$};
				\node(v3)[scale=0.8] at (1,2.8) {$h_{2}$};
				\node(phi3)[scale=0.8] at (1,-0.8) {$\phi_{d}$};
				
				\draw[->] (phi2) -- (phi1);
				\draw[->] (v1) -- (v2);
				\draw[->] (E') -- (E);
				\draw[->] (C') -- (C);
				\draw[->] (C') -- (E');
				\draw[->] (C) -- (E);
				\draw[->] (j2) -- (j1);
				\draw[->] (j1') -- (j2');
				
				\node(D2) at (5,1) {($D_{2}$)};
				
				\end{tikzpicture}
			\end{center}
		\end{figure}
		
%
%
%
%
%
%
%
%
		\end{lem} 
		
		\begin{proof}[Proof of Lemma \ref{diag_com}]
		
		As $g=0$, $X$ and $\mathbb{P}^{1}(C)$ are birationally equivalent (see \cite[IV,
		.1.3.5]{Hart}). By \cite[I.4.5]{Hart}, $E$ and $C(z)$ are isomorphic as $C$-algebras. Hence
		$$E=C(u),$$
		where $u\in E$ is transcendental over $C$.
		Now, we are going to express the morphisms $h_{1}$ and $h_{2}$ of diagram $(D_{2})$ with respect to the function fields morphisms associated with the morphisms of curves of diagram (D). As to start, applying Hurwitz formula \eqref{hurwitz} to $f$, we get that the sets $\varphi^{-1}(0)$ and $\varphi^{-1}(\infty)$ have respectively exactly one element, denoted $a$ and $b$, and that the following property is satisfied.
		
		\begin{enumerate}\setcounter{enumi}{3}
			\item
			$f$ is unramified outside $\{a,b\}$.
		\end{enumerate}
		 
		On the other hand, we notice that the following properties characterize the morphism of curves $\tilde{h_{2}} : X \longrightarrow X$ associated with the function field morphism $h_{2}$ of diagram $(D_{2})$ (we identify $0,\infty$ of $\mathbb{P}^{1}(C)$ with the corresponding elements of $X$ via birational equivalence).
		
		\begin{enumerate}[(i)]
			\item
			$\tilde{h_{2}}(0)=0$, $\tilde{h_{2}}(\infty)=\infty$.
			\item
			$\tilde{h_{2}}$ has degree $d$.
			\item
			$\tilde{h_{2}}$ is totally ramified at $0$ and $\infty$. 
			\item
			$\tilde{h_{2}}$ is unramified outside $\{0,\infty\}$.
		\end{enumerate}	
		
		These assertions are exactly the assertions 1-4 satisfied by $f$, except that $\{a,b\}$ is replaced by $\{0,\infty\}$. To correct it, we consider an automorphism $c$ of $X$ such that
		
		\begin{align}
		c : X &\rightarrow X\\
		a&\mapsto 0\\
		b&\mapsto \infty.
		\end{align}

	From now on, if $h$ is a morphism of curves, $h^{*}$ denotes the associated morphism of function fields. We deduce from properties 1-4 of $f$ that the morphism  $cfc^{-1}$ satisfies properties (i)-(iv). Hence $h_{2}=(cfc^{-1})^{*}$. Now, let $h_{1}=\left(\varphi c^{-1}\right)^{*}$ and $P(u)\in C(u)$ be such that $\left(\varphi c^{-1}\right)^{*}(z)=P(u)$. 
	
	By Diagram (D), we get the following commutative diagram.
	
	 \begin{figure}[H]
	 	\begin{center}
	 		\begin{tikzpicture} 
	 		
	 		\node(X) at (0,2) {$X$}; 
	 		\node(X') at (2,2) {$X$}; 
	 		
	 		\node(p1) at (0,0) {$\mathbb{P}^{1}(C)$};
	 		\node(p1') at (2,0) {$\mathbb{P}^{1}(C)$};
	 		\node(g1)[scale=0.8] at (0,-0.6) {$(x_{1},x_{2})$};
	 		\node(g2)[scale=0.8] at (2.25,-0.6) {$\left(1,\left(\frac{x_{2}}{x_{1}}\right)^{d}\right)$};
	 	
	 		\node(j)[scale=0.7] at (-0.5,1.2) {$\varphi c^{-1}$};
	 		\node(j')[scale=0.7] at (2.5,1.2) {$\varphi c^{-1}$};
	 		\node(cfc)[scale=0.7] at (1,2.3) {$cfc^{-1}$};
	 		\node(g3)[scale=0.7] at (1,-0.8) {$j$};
	 		
	 		\draw[->] (g1) -- (g2);
	 		\draw[->] (X) -- (X');
	 		\draw[->] (p1) -- (p1');
	 		\draw[->] (X) -- (p1);
	 		\draw[->] (X') -- (p1');
	 		
	 		\end{tikzpicture}
	 	\end{center}
	 \end{figure}

	Lemma \ref{diag_com} follows by considering the associated morphisms of function fields.
\end{proof}

	We are now able to conclude the proof of Proposition \ref{l_prop}. Let us read Diagram $(D_{2})$. On the one hand, we obtain
		$$h_{2}\circ h_{1}(z)=P(u^{d}),$$
		and on the other hand
		$$h_{1}\circ\phi_{d}(z)=P(u)^{d}.$$
		Then $P(u^{d})=P(u)^{d}$. But $p\nmid d$. Hence 
		$$P(u)=\lambda u^{N},$$
		where $N\in \mathbb{Z}$ and $\lambda^{d}=\lambda\in C$.
		
		Now, let $c_{1}=\left(c^{-1}\right)^{*}$ denote the function field automorphism associated with $c^{-1}$. We have
		$$h_{1}=\left(\varphi c^{-1}\right)^{*}=c_{1}\varphi^{*}=c_{1}i.$$
		Let us set
		$$i(z)=z=Q(u),$$
		where $Q(u)\in C(u)$.
		Then, we get
		\begin{align*}
		\lambda u^{N} &=h_{1}(z)\\
		&= c_{1}i(z)\\
		&=c_{1}(Q(u)).
		\end{align*}
		Hence
		\begin{align*}
		z&=Q(u)\\
		&=c_{1}^{-1} \left(\lambda u^{N}\right)\\
		&=\lambda \left(c_{1}^{-1}(u)\right)^{N}.
		\end{align*}
		Now, let $\mu$ be a $N$-th root of $\lambda$ in $C$. Let us set
		$$\begin{cases} z_{N}&=\mu c_{1}^{-1}(u) \text{, if } N\geq 0\\
		&=1/(\mu c_{1}^{-1}(u)) \text{ otherwise}.\end{cases}$$
		In both cases, $z_{N}\in E$ and we obtain
		$$z=z_{N}^{|N|},$$
	    and
		$$E=C(u)=C(z_{N}).$$
		Finally, as $E$ is the separable closure of $C(z)$ in $L$, $L$ is a purely inseparable extension of $C(z_{N})$. Proposition \ref{l_prop} is proved. 
		\end{proof}
		We deduce the analogue of Corollary 5 of \cite{JR}.
		
		\begin{cor}
			\label{l_cor}
			Let $L\subset C((z))$ be a finite extension of $C(z)$. Let $d\geq 2$ be an integer such that $p\nmid d$. We assume that the endomorphism $\phi_{d}$ of $C(z)$ defined by $\phi_{d}(P(z))=P(z^{d})$ extends to a field endomorphism of $L$. Then, $L$ is a purely inseparable extension of $C(z)$.
		\end{cor}
		
		\begin{proof}[Proof of Corollary \ref{l_cor}]
		We have $z_{N}^{N}=z$, with $z_{N}\in L\subset C((z))$. Hence, $N=1$, $z_{N}=z$ and Corollary \ref{l_cor} is proved.	
		\end{proof}
		
		We are now able to prove Theorem \ref{l_mahler_bis}. To do so, we use here Cartier operators (see Proposition \ref{Cartier}).
		
\begin{proof}[Proof of Theorem \ref{l_mahler_bis}]
		
Let us assume that $f(z)$ is algebraic over $\overline{\mathbb{K}}(z)$. Let \eqref{eq_inh_min} be the minimal inhomogeneous equation of $f(z)$ and set 
$$L=C(z)\left(f(z), \ldots, f\left(z^{d^{m-1}}\right)\right).$$
We note that, for all $l\geq 0$, $f\left(z^{d^{l}}\right)$ is algebraic over $\overline{\mathbb{K}}(z)$. It follows that $L$ is a finite extension of $C(z)$.
Besides, $d$-Mahler Equation \eqref{eq_inh_min} guarantees that $\phi_{d}$ induces a field endomorphism of $L$. Then, by Corollary \ref{l_cor}, $L$ is a purely inseparable extension of $C(z)$. We deduce that there exists an integer $s$ such that $f(z)^{p^{s}}\in C(z)$. Hence, there exist non-zero polynomials $A(z), B(z)\in C[z]$ such that
\begin{equation}
\label{B_cartier}
B(z)f(z)^{p^{s}}=A(z)
\end{equation}
	
Now, let us choose for every integer $j\in\{1,\ldots, s\}$ a Cartier operator $\Lambda^{(j)}$ such that $$\Lambda^{(s)}\circ\cdots\circ\Lambda^{(1)}(B(z))\neq 0.$$
We apply $\Lambda:=\Lambda^{(s)}\circ\cdots\circ\Lambda^{(1)}$ to \eqref{B_cartier} and get
$$\Lambda(B(z))f(z)=\Lambda(A(z)).$$
Then, $f(z)\in C(z)\cap \mathbb{K}\{z\}=\mathbb{K}(z)$ and Theorem \ref{l_mahler_bis} is proved.		
\end{proof}

	\begin{proof}[Proof of Corollary \ref{l_mahler}]
	Let $L=\overline{\mathbb{K}}(z)(f_{1}(z), \ldots, f_{n}(z))$ be a $d$-Mahler extension over $\overline{\mathbb{K}}(z)$. Without any loss of generality, we can assume that $(f_{1}(z), \ldots, f_{n}(z))$ is a solution vector of System \eqref{syst_gen}. Indeed, if not, we can insert the $f_{i}(z)$ into a solution vector $$\underline{g}(z)=(f_{1}(z), \ldots, f_{n}(z), g_{n+1}(z),\ldots, g_{N}(z))$$ of a $d$-Mahler system of size $N\geq n$. Then, we have
	$$L\subseteq\overline{\mathbb{K}}(z)(\underline{g}(z)).$$ 
	Thus, if we prove that $\overline{\mathbb{K}}(z)(\underline{g}(z))$ is regular over $\overline{\mathbb{K}}(z)$, it follows from \cite[Corollary A1.6]{Eisenbud} that $L$ is regular over $\overline{\mathbb{K}}(z)$. Now, by Proposition \ref{tjrs_sep}, $L$ is separable over $\overline{\mathbb{K}}(z)$. It thus remains to prove that every element of $L$ that is algebraic over $\overline{\mathbb{K}}(z)$ belongs to $\overline{\mathbb{K}}(z)$. To do so, we follow the same approach as in \cite{A-F}.
	Let $E$ be the algebraic closure of $\overline{\mathbb{K}}(z)$ in $L$ and $f(z)\in E$. Our aim is to prove that $f(z)$ is $d$-mahler and apply Theorem \ref{l_mahler_bis}. First, it follows from System \eqref{syst_gen} that for all $l\geq 0$, $f(z^{d^{l}})\in L$. Then, the fact that $f(z)\in E$ implies that for all $l\geq 0$, $f(z^{d^{l}})\in E$. Now, it suffices to prove that $E$ is a finite extension of $\overline{\mathbb{K}}(z)$.
	As $L$ is a finitely generated $\overline{\mathbb{K}}(z)$-algebra, the sub-extension $E$ has the same property (see for example \cite[VIII, Exercise 4]{Lang}). But $E$ is also an algebraic extension of $\overline{\mathbb{K}}(z)$. Hence, $E$ is a finite extension of $\overline{\mathbb{K}}(z)$. It follows that $f(z)$ is $d$-mahler. Thus, by Theorem \ref{l_mahler_bis}, $f(z)\in \mathbb{K}(z)$ and Corollary \ref{l_mahler} is proved.
	\end{proof} 

\section{Proof of Corollary \ref{cor_important}}
\label{cor}
We prove here Corollary \ref{cor_important}.

\begin{proof}[Proof of Corollary \ref{cor_important}]
Let us keep the assumptions of Corollary \ref{cor_important}. Let \eqref{syst_inh} be the minimal inhomogeneous system satisfied by $f(z)$. Let us prove the first assertion. Let us assume that $f(\alpha)\in\overline{\mathbb{K}}$. For all $i\in\{1,\ldots,m\}$, let us set
$$f_{i}(z)=f(z^{d^{i-1}}).$$
There exists an integer $l\geq 0$ such that $\alpha^{d^{l}}$ is regular for $d$-Mahler System \eqref{syst_inh}. Therefore, by Theorem \ref{phil_p} and minimality of \eqref{syst_inh}, the numbers $1,f_{1}(\alpha^{d^{l}}),\ldots, f_{m}(\alpha^{d^{l}})$ are linearly independent over $\overline{\mathbb{K}}$. Moreover, we can write
\begin{equation}
\label{syst_bis}
\begin{pmatrix}
1\\f_{1}(z) \\ \vdots \\ f_{m}(z) 
\end{pmatrix} =A_{l}(z)\begin{pmatrix}
1\\f_{1}(z^{d^{l}}) \\ \vdots \\ f_{m}(z^{d^{l}}) 
\end{pmatrix},
\end{equation}
where 
$$A_{l}(z)=A^{-1}(z)A^{-1}(z^{d})\ldots A^{-1}\left(z^{d^{l-1}}\right).$$
Then, we can see that $\alpha$ is not a pole of the matrix $A_{l}(z)$ of System \eqref{syst_bis}. Indeed, otherwise, let $r$ denote the maximum of the order of $\alpha$ as a zero of the denominators of the coefficients of $A_{l}(z)$. We get 
\begin{equation}
\label{mat_ss_pole}
(z-\alpha)^{r}\begin{pmatrix}
1\\f_{1}(z) \\ \vdots \\ f_{m}(z) 
\end{pmatrix} =B_{l}(z)\begin{pmatrix}
1\\f_{1}(z^{d^{l}}) \\ \vdots \\ f_{m}(z^{d^{l}}) 
\end{pmatrix},
\end{equation}
where $B_{l}(z)=(z-\alpha)^{r}A_{l}(z)$ has no pole at $\alpha$ and is such that $B_{l}(\alpha)\neq 0$. Then, setting $z=\alpha$ in \eqref{mat_ss_pole}, we would find a linear non-trivial relation between the numbers $1,f_{1}(\alpha^{d^{l}}),\ldots, f_{m}(\alpha^{d^{l}})$ which contradicts the fact that they are linearly independent over $\overline{\mathbb{K}}$. Now, if we set
$$\Lambda=\begin{pmatrix}
f(\alpha) & -1 & 0 & \ldots & 0
\end{pmatrix},$$
we obtain
\begin{align*}
0&=\Lambda\begin{pmatrix}
1\\f_{1}(\alpha) \\ \vdots \\ f_{m}(\alpha) 
\end{pmatrix}\\
&=\Lambda A_{l}(\alpha)\begin{pmatrix}
1\\f_{1}(\alpha^{d^{l}}) \\ \vdots \\ f_{m}(\alpha^{d^{l}})\end{pmatrix}. 
\end{align*}
Then, the fact that $1,f_{1}(\alpha^{d^{l}}),\ldots, f_{m}(\alpha^{d^{l}})$ are linearly independent over $\overline{\mathbb{K}}$ implies that
\begin{equation}
\label{rel_mat}
\Lambda A_{l}(\alpha)=0.
\end{equation}
Now, as the first coordinate of the solution vector of System \eqref{syst_bis} is $1$, there exists a column vector $\begin{pmatrix}
u_{0}\\u_{1} \\ \vdots \\ u_{m} 
\end{pmatrix}\in\mathbb{K}(\alpha)^{m+1}$ of $A_{l}(\alpha)$ such that $u_{0}\neq 0$. Then, by \eqref{rel_mat} we get
$$f(\alpha)=\frac{u_{1}}{u_{0}}\in\mathbb{K}(\alpha).$$

Let us prove the second statement. If $\alpha$ is a regular number for System \eqref{syst_inh}, let us assume by contradiction that $f(\alpha)$ is algebraic over $\overline{\mathbb{K}}$, that is $f(\alpha)\in\overline{\mathbb{K}}$. Then, the numbers $1,f(\alpha)$ are linearly dependent over $\overline{\mathbb{K}}$ and hence, the numbers $1,f(\alpha),\ldots, f\left(\alpha^{d^{m-1}}\right)$ are linearly dependent over $\overline{\mathbb{K}}$. Then,  Theorem \ref{phil_p} implies that there exists a linear relation between the functions $1,f(z), \ldots, f\left(z^{d^{m-1}}\right)$ over $\overline{\mathbb{K}}(z)$. This contradicts the minimality of Equation \eqref{eq_inh_min} and proves that $f(\alpha)$ is transcendental over $\overline{\mathbb{K}}$.

\end{proof}

\begin{rem}
\em{
	In order to prove the first statement of Corollary \ref{cor_important}, we showed the existence of an integer $l\geq 0$ and a matrix $B(z)\in\text{GL}_{m+1}(\mathbb{K}(z))$ such that
	\begin{equation}
	\label{syst_rem}
	\begin{pmatrix}
	1\\f_{1}(z) \\ \vdots \\ f_{m}(z) 
	\end{pmatrix} =B(z)\begin{pmatrix}
	1\\f_{1}(z^{d^{l}}) \\ \vdots \\ f_{m}(z^{d^{l}}) 
	\end{pmatrix},
	\end{equation}
and
\begin{enumerate}
	\item 
	the number $\alpha^{d^{l}}$ is regular for System \eqref{syst_rem},
	\item
	the number $\alpha$ is not a pole of $B(z)$.
\end{enumerate}
This is the analogue of \cite[Theorem 1.5]{Beukers} for $E$-functions and of \cite[Théorème 1.10]{A-F} in the particular case of linearly independent Mahler functions.}
\end{rem}

\section{Examples}
\label{Examples}
In this section, we illustrate Theorem \ref{phil_p} and provide examples, in the case where $p\mid d$, of regular and non-regular $d$-Mahler extensions.

\subsection{An application of Theorem \ref{phil_p}}

Let $d$ be an integer such that $p\nmid d$. 
Let us consider the system

\begin{equation}
\label{syst_ex_1}
\begin{pmatrix}
f_{1}(z^{d}) \\ f_{2}(z^{d}) \\ f_{3}(z^{d})\end{pmatrix}\begin{pmatrix}
1 & 0 & 0 \\
0 & 0 & 1\\
z^{d^{2}}-z & 1 & -1
\end{pmatrix}\begin{pmatrix}
f_{1}(z) \\ f_{2}(z) \\ f_{3}(z)\end{pmatrix}.
\end{equation}

Let us set
\begin{equation}
\label{definition_a}
a(z)=z+\sum_{n=1}^{+\infty}F_{n}z^{d^{n}},
\end{equation}
where $F_{n}$ is the residue modulo $d$ of the $n$-th Fibonacci number (with $F_{1}=1, F_{2}=1$). Then, System \eqref{syst_ex_1} is given by
\begin{equation}
\label{l_a}
f_{1}(z)=1, f_{2}(z)=a(z), f_{3}(z)=a(z^{d}).
\end{equation}

 By Corollary \ref{l_mahler}, the $d$-Mahler extension $\mathcal{E}=\overline{\mathbb{K}}(z)(f_{1}(z),f_{2}(z),f_{3}(z))$ is regular over $\overline{\mathbb{K}}(z)$.
The advantage of Theorem \ref{phil_p} is that we do not have to study algebraic relations between $f_{1}(z)$, $f_{2}(z)$, $f_{3}(z)$ to get the following result.

\begin{prop}
	\label{illustr_ind_lin}
	Let $\alpha\in\overline{\mathbb{K}}$, $0<|\alpha|<1$. Then, $1,a(\alpha),a(\alpha^{d})$ are linearly independent over $\overline{\mathbb{K}}$. 
\end{prop}

By Corollary \ref{cor_direct}, all we have to prove is the following result.

\begin{lem}
	\label{l_ind_lin_ex}
	The functions $f_{1}(z),f_{2}(z),f_{3}(z)$ are linearly independent over $\overline{\mathbb{K}}(z)$.
\end{lem}

Proof of Lemma \ref{l_ind_lin_ex} does not involve difficult arguments and illustrates the interest of Theorem \ref{phil_p}.

\begin{proof}[Proof of Lemma \ref{l_ind_lin_ex}]
	Let us assume by contradiction that there exist coprime polynomials \\ $P_{-1}(z), P_{0}(z), P_{1}(z)\in\overline{\mathbb{K}}[z]$ such that
	$$P_{-1}(z)+P_{0}(z)a(z)+P_{1}(z)a(z^{d})=0.$$
	It follows that
	\begin{equation}
	\label{fibo_deg_grand}
	P_{-1}(z)+P_{0}(z)z+P_{0}(z)\sum_{n=1}^{+\infty}F_{n}z^{d^{n}}+P_{1}(z)z^{d}+P_{1}(z)\sum_{n=2}^{+\infty}F_{n-1}z^{d^{n}}=0.
	\end{equation}
	For all $n\geq 1$, let us set $a_{n}=d^{n}-d^{n-1}$. Then, the sequence $(a_{n})_{n}$ is strictly increasing. Now, let us take $N\in\mathbb{N}$ such that $a_{n}>\max(\deg(P_{i}))$ for all $n\geq N$. If we compare the coefficients of $z^{d^{N}}$ and $z^{d^{N+1}}$ respectively between the left and right-hand side of \eqref{fibo_deg_grand}, we get
	\begin{equation}
	\label{fibo_det}
	\begin{cases}
	p_{0,0}F_{N}+p_{1,0}F_{N-1}&=0\\
	p_{0,0}F_{N+1}+p_{1,0}F_{N}&=0,
	\end{cases} 
	\end{equation}
	where $p_{0,0}, p_{1,0}$ are respectively the constant term of $P_{0}(z)$ and $P_{1}(z)$.
	By property of the Fibonacci sequence, the determinant $F_{N}^{2}-F_{N-1}F_{N+1}$ of System \eqref{fibo_det} is equal to $(-1)^{N+1}\neq 0$. Hence, $p_{0,0}=p_{1,0}=0$. But, by \eqref{fibo_deg_grand}, the constant term of $P_{-1}(z)$ is equal to zero. This contradicts the fact that the $P_{i}$'s are coprime. Lemma \ref{l_ind_lin_ex} is proved.
\end{proof}

\subsection{Regular extensions}
\label{sub_phil}

If $f_{1}(z), \ldots, f_{n}(z) \in \mathbb{K}\{z\}$ are algebraically independent functions over $\overline{\mathbb{K}}(z)$, then, the extension $\overline{\mathbb{K}}(z)(f_{1}(z), \ldots, f_{n}(z))$ is regular over $\overline{\mathbb{K}}(z)$. Indeed, let us set $\mathcal{E}=\overline{\mathbb{K}}(z)(f_{1}(z), \ldots, f_{n}(z))$. As $(f_{1}(z), \ldots, f_{n}(z))$ is a transcendence basis of $\mathcal{E}$ over $\overline{\mathbb{K}}(z)$, $\mathcal{E}$ is separable over $\overline{\mathbb{K}}(z)$. Moreover, let us assume that there exists an element $$a(z)\in\overline{\mathbb{K}}(z)(f_{1}(z),\ldots,f_{n}(z))\cap \overline{\mathbb{K}(z)}\setminus\overline{\mathbb{K}}(z).$$
Then, by \eqref{relfct}, the functions $f_{1}(z),\ldots,f_{n}(z)$ are algebraically dependent over $\overline{\mathbb{K}(z)}$, and hence, over $\overline{\mathbb{K}}(z)$, which is a contradiction.
\medskip

This shows in particular that, when $p \mid d$, there exist regular $d$-Mahler extensions. In fact, in this case, there even exist regular $d$-Mahler extensions associated with a solution of a $d$-Mahler system whose coordinates are algebraically dependent over $\overline{\mathbb{K}}(z)$. Indeed, let us consider the system

\begin{equation}
\label{syst_fg_ex}
\begin{pmatrix} f_{1}(z^{q}) \\ f_{2}(z^{q}) \\ f_{3}(z^{q}) \\  f_{4}(z^{q}) \end{pmatrix}=
\begin{pmatrix} 1 & 0 & 0 & 0 \\ 
\left(\frac{1}{z}\right)^{q}-T & -\left(\left(\frac{1}{z}\right)^{q}-T\right) & 0 & 0 \\
0 & 0 & \frac{1}{1-Tz^{q}} & 0\\
0 & 0 & \frac{1}{z^{q}} & -\frac{1}{z^{q}}
\end{pmatrix}\begin{pmatrix}
f_{1}(z) \\ f_{2}(z) \\ f_{3}(z) \\  f_{4}(z)
\end{pmatrix}.
\end{equation}

Now, let us set

$$f(z)=\sum_{n=0}^{+\infty} \frac{(-1)^{n}}{(\left(\frac{1}{z}\right)^{q}-T)(\left(\frac{1}{z}\right)^{q^{2}}-T)\cdots(\left(\frac{1}{z}\right)^{q^{n}}-T)},$$
and
$$g(z)=\prod_{n=1}^{+\infty} \left(1-Tz^{q^{n}} \right).$$

These functions are respectively $f_{1}(\frac{1}{z})$ and $g(z)$ introduced by L. Denis in \cite{Denis_log, D3}. According to L. Denis, $f(z)$ and $g(z)$ are analytic in $\{z\in C, |z|<\left(\frac{1}{q}\right)^{\frac{1}{q}}\}$ (even in the open unit disk for $g(z)$) and algebraically independent over $\overline{\mathbb{K}}(z)$. 

We see that a solution to \eqref{syst_fg_ex} is given by
$$f_{1}(z)=1, f_{2}(z)=f(z), f_{3}(z)=g(z), f_{4}(z)=f(z)g(z).$$

Then, the Mahler extension $\overline{\mathbb{K}}(z)(f_{1}(z),\ldots,f_{4}(z))$
is regular over $\overline{\mathbb{K}}(z)$. Indeed, $f(z), g(z)$ are algebraically independent over $\overline{\mathbb{K}}(z)$ and $\overline{\mathbb{K}}(z)(1,f(z),g(z),f(z)g(z))=\overline{\mathbb{K}}(z)(f(z),g(z))$.

\subsection{Non-regular extensions}
\label{non_reg_ext}

We have seen in the introduction of this paper that the $p$-Mahler extension $\mathcal{E}=\overline{\mathbb{K}}(z)(1,\sum_{n=0}^{+\infty}z^{p^{n}})$ is not regular over $\overline{\mathbb{K}}(z)$. In this case, $\mathcal{E}$ was an algebraic extension of $\overline{\mathbb{K}}(z)$. But there also exist non-regular transcendental $q$-Mahler extensions. Moreover, such an extension can be found among the simplest possible Mahler extensions, that is those of the form

$$\mathcal{E}=\overline{\mathbb{K}}(z)(f(z))_{\sigma_{q}},$$
where $f(z)$ is a transcendental $q$-Mahler function. 

Let us set 
$$f(z)=\sum_{n=0}^{+\infty} \frac{z^{q^{n}}}{1-Tz^{q^{n}}}-z-\sum_{n=1}^{+\infty}F_{n}z^{q^{n}},$$

where $F_{n}$ is the residue modulo $q$ of the $n$-th Fibonacci number (with $F_{1}=1, F_{2}=1$). By \cite{D3}, $f(z)$ is a transcendental analytic function in $\{z\in C, |z|<\frac{1}{q}\}$. Moreover, we have

\begin{equation}
\label{eq_inh_min_ex_3}
f\left(z^{q^{3}}\right)-2f\left(z^{q}\right)+f(z)-R(z)=0,
\end{equation}
where $R(z)=z-z^{q}-z^{q^{2}}+z^{q^{3}}-\frac{z}{1-Tz}+\frac{z^{q}}{1-Tz^{q}}+\frac{z^{q^{2}}}{1-Tz^{q^{2}}}$.

We prove the following proposition.

\begin{prop}
	\label{prop_ex_non_reg}
The $q$-Mahler extension $\mathcal{E}=\overline{\mathbb{K}}(z)(f(z))_{\sigma_{q}}$	is non-regular over $\overline{\mathbb{K}}(z)$.
\end{prop}

\begin{proof}[Proof of Proposition \ref{prop_ex_non_reg}]
	
Let us set $$g(z)=\sum_{n=0}^{+\infty} \frac{z^{q^{n}}}{1-Tz^{q^{n}}},$$  and $$a(z)=z+\sum_{n=1}^{+\infty}F_{n}z^{q^{n}}.$$ First, let us notice that $g(z^{q})=g(z)-\frac{z}{1-Tz}$, $a(z^{q^{2}})=-a(z^{q})+a(z)+z^{q^{2}}-z$, and $a(z)^{q^{2}}=-a(z)^{q}+a(z)+z^{q^{2}}-z$. On the other hand, as the sequence
$$a_{1}=1,\text{ } a_{k}=\begin{cases}F_{n}, &\text{ if } k=q^{n}\\
 0 &\text{ otherwise}\end{cases}, \text{ for } k\geq 2$$
is not ultimately periodic, we obtain that
\begin{equation*}
a(z)\in\overline{\mathbb{K}(z)}\setminus\overline{\mathbb{K}}(z).
\end{equation*}

Now, we compute

\begin{align}
f(z)-f(z^{q})&=a(z)-g(z)-a(z^{q})+g(z^{q}) \nonumber\\
&=a(z)-a(z^{q})-g(z)+g(z)-\frac{z}{1-Tz}, \nonumber\\
&=a(z)-a(z^{q})-\frac{z}{1-Tz}\nonumber\\
\label{rel_lin_g_alpha}
&=a(z^{q^{2}})-z^{q^{2}}+z-\frac{z}{1-Tz}.
\end{align}

But $a(z^{q^{2}})\in\overline{\mathbb{K}(z)}\setminus\overline{\mathbb{K}}(z)$ (if not, apply suitable Cartier operators and get $a(z)\in\overline{\mathbb{K}}(z)$ which is a contradiction). This implies that $f(z)-f\left(z^{q}\right)\in\mathcal{E}\cap\overline{\mathbb{K}(z)}\setminus\overline{\mathbb{K}}(z)$  and proves that $\mathcal{E}$ is not regular over $\overline{\mathbb{K}}(z)$. Proposition \ref{prop_ex_non_reg} is proved.
\end{proof}

Besides, in this case, a direct and elementary approach shows that the conclusion of Theorem \ref{phil_p} is not satisfied. First, let us prove the following lemma.

\begin{lem}
	\label{g_lin_ind}
	The functions $1,f(z),f\left(z^{q}\right),f\left(z^{q^{2}}\right)$ are linearly independent over $\overline{\mathbb{K}}(z)$. In other words, inhomogeneous Equation \eqref{eq_inh_min_ex_3} is minimal.	
\end{lem}

\begin{proof}
	Let us assume by contradiction that there exist polynomials $P_{-1}(z),\ldots,P_{2}(z)\in\overline{\mathbb{K}}[z]$ not all zero such that
$$P_{-1}(z)+P_{0}(z)f(z)+P_{1}(z)f\left(z^{q}\right)+P_{2}(z)f\left(z^{q^{2}}\right)=0$$
Then, after computations, we get
\begin{multline}
\label{eg_f_transc}
g(z)(P_{0}(z)+P_{1}(z)+P_{2}(z))=P_{-1}(z)+P_{1}(z)P(z)+P_{2}(z)Q(z)+P_{2}(z)P\left(z^{q}\right)+P_{2}(z)P(z)\\+a(z)(P_{0}(z)+P_{2}(z))+a\left(z^{q}\right)(P_{1}(z)-P_{2}(z)),
\end{multline}
where $P(z)=\frac{z}{1-Tz}$ and $Q(z)=z^{q^{2}}-z$. But $g(z)$ is transcendental over $\overline{\mathbb{K}}(z)$ \cite{D3}, whereas the right-hand side of \eqref{eg_f_transc} is algebraic over $\overline{\mathbb{K}}(z)$. Hence
 
\begin{equation}
\label{P_i}
P_{0}(z)+P_{1}(z)+P_{2}(z)=0,
\end{equation} 

and

\begin{multline}
\label{a_lin_ind}
P_{-1}(z)+P_{1}(z)P(z)+P_{2}(z)Q(z)+P_{2}(z)P\left(z^{q}\right)+P_{2}(z)P(z)+a(z)(P_{0}(z)+P_{2}(z))\\+a\left(z^{q}\right)(P_{1}(z)-P_{2}(z))=0.
\end{multline}

Now, let us notice that the function $a(z)$ seems similar to the one defined by \eqref{definition_a}. But in \eqref{definition_a}, $a(z)$ is $d$-Mahler with $p\nmid d$ and is transcendental over $\overline{\mathbb{K}}(z)$, whereas here, $a(z)$ is $q$-Mahler, with $q=p^{r}$ and algebraic over $\overline{\mathbb{K}}(z)$. Nevertheless, arguing as in Lemma \ref{l_ind_lin_ex}, we get that $1,a(z),a(z^{q})$ are also linearly independent over $\overline{\mathbb{K}}(z)$. Hence, by \eqref{a_lin_ind}
$$P_{1}(z)=P_{2}(z)=-P_{0}(z).$$
By \eqref{P_i}, we get $P_{0}(z)=P_{1}(z)=P_{2}(z)=0$. Finally, by \eqref{a_lin_ind}, $P_{-1}(z)=0$, which is a contradiction. Lemma \ref{g_lin_ind} is proved.
\medskip
\end{proof}

Finally, let $\alpha\in\overline{\mathbb{K}}$, $0<|\alpha|<1$ be a regular number for the system associated with \eqref{eq_inh_min_ex_3}, that is, $\alpha\notin\left\{\left(\frac{1}{T}\right)^{1/q^{k}}\right\}_{k\geq 0}$. By \eqref{rel_lin_g_alpha}, we see that $f(\alpha)-f(\alpha^{q})\in\overline{\mathbb{K}}$, that is $1,f(\alpha),f(\alpha^{q})$ are linearly dependent over $\overline{\mathbb{K}}$. Hence, it follows from Lemma \ref{g_lin_ind} that the conclusion of Theorem \ref{phil_p} is not satisfied.
\bigskip

\textbf{Aknowledgements}. The author would like to thank Patrice Philippon for all the interesting discussions they had, which go beyond the framework of this paper, his availability and his kindness; Julien Roques for his review of a preliminary version of the proof of Proposition \ref{l_prop}; and, obviously, Boris Adamczewski for his multiple reviews and corrections of this paper and his useful advice. Any critical remark must be exclusively addressed to the author of this paper.

\bibliographystyle{smfplain}
\bibliography{Theoreme_patrice_philippon_car_p}
\printindex

\nocite{Nest_shid}
\nocite{KrullI}

\end{document}